%% file: main.tex
\documentclass[11pt,preprint]{amsart}
\usepackage[T1]{fontenc}
\usepackage[utf8]{inputenc}
\usepackage{lmodern}
 \usepackage[margin=28mm,bindingoffset=0mm]{geometry}
\usepackage{enumerate}
\usepackage{amsmath}
\usepackage{amsfonts}
\usepackage{amssymb}
\usepackage{amsthm}
\usepackage{enumitem}
\usepackage{version}
\usepackage{graphicx}

\usepackage[all]{xy}
\usepackage{bm}
\usepackage{hyperref}
\usepackage{appendix}

\author{Alexis Garcia}
\address{Université d'Avignon}
\email{alexis.garcia@univ-avignon.fr}
\title{On the classification of meromorphic affine connections on complex compact surfaces}
\makeatletter
\renewcommand{\email}[2][]{%
  \ifx\emails\@empty\relax\else{\g@addto@macro\emails{,\space}}\fi%
  \@ifnotempty{#1}{\g@addto@macro\emails{\textrm{(#1)}\space}}%
  \g@addto@macro\emails{#2}%
}
\makeatother

\subjclass{51, 53}
\usepackage{lipsum}

\newtheoremstyle{defstyle} % name
    {1em}                    % Space above
    {1em}                    % Space below
    {\small}                   % Body font
    {}                           % Indent amount
    {\bfseries}                   % Theorem head font
    {.}                          % Punctuation after theorem head
    {.5em}                       % Space after theorem head
    {}  % Theorem head spec (can be left empty, meaning ?normal?)

\newtheorem*{definition*}{Definition}
\newtheorem*{theorem*}{Theorem}
\newtheorem{theorem}{Theorem}[section]
\newtheorem{lemma}{Lemma}[section]

\newtheorem{proposition}{Proposition}[theorem]

\newtheorem{corollary}{Corollary}[section]

\theoremstyle{defstyle}
\newtheorem{definition}{Definition}[section]

\newcommand{\der}[2]{\frac{\partial #1}{\partial#2}}

\makeatletter
\renewcommand*\env@matrix[1][*\c@MaxMatrixCols c]{%
  \hskip -\arraycolsep
  \let\@ifnextchar\new@ifnextchar
  \array{#1}}
\makeatother

\date{\today}
\begin{document}

\maketitle

\tableofcontents

\input{corpus.tex}

\bibliographystyle{plain}

\bibliography{biblio}

\end{document}

%% file: corpus.tex
\section{Introduction}
\subsection{Uniformization of complex compact manifolds and meromorphic geometric structures}
The question of classifying compact complex manifolds, by means of geometric and topological invariant, is an old and still active one. The first historical example of an answer is the Riemann's uniformization theorem for complex compact manifolds of dimension one (i.e. Riemann surfaces or compact complex  smooth curves). In particular, it asserts that any such a complex manifolds is the quotient of an open subset $U$ of: either the projective curve $\mathbb{P}^1$, the complex line $\mathbb{C}$, or the Poincaré disk $\mathbb{H}$, by a discrete subgroup $\Gamma$ of projective transformations preserving $U$ with free left action.

This is an example of \textit{uniformizability} in the sense that any complex compact manifold of dimension one is the quotient $\Gamma\backslash U$ of an open set $U\subset X$ of a fixed space $X$ by a subgroup $\Gamma$ of automorphisms of $X$ preserving a fixed \textit{geometric structure} and with a free action. The notion of \textit{geometric structures} will not be defined in this paper but may be thought as one of these examples : a holomorphic affine connection (see \autoref{affineconnection}), a holomorphic projective connection (see for example \cite{KobayashiOchiai}) or a holomorphic reduction of a $k$-th order frame bundle of $X$.

\subsection{Meromorphic affine connections on surfaces}
In \cite{InoueKobayashi}, Inoue, Kobayashi and Ochiai classified holomorphic affine connections on compact complex surfaces. It turns out that any complex compact surfaces admitting such a geometric structure is also equipped with a flat one, that is a holomorphic affine structure. Many of them are quotients of an open subset of $\mathbb{C}^2$ by affine transformations. This result was completed by Kobayashi and Ochiai in \cite{KobayashiOchiai}, where it appears that any complex compact surface endowed with a holomorphic projective structure is uniformizable by the unit ball in $\mathbb{C}^2$.

It is thus natural to investigate which complex compact manifold can be endowed with a particular type of holomorphic geometric structure. In a recent paper \cite{BDM}, Biswas, Dumitrescu and McKay gave rather general classification result, asserting that many holomorphic geometric structures (in particular holomorphic affine connections) can't be beared by simply connected compact complex manifolds with constant meromorphic functions.  

We may also ask whether allowing the geometric structure to admits some reasonable singularities (namely poles) could enable to encompass more compact complex manifolds, and investigate for a definition of meromorphic versions of the uniformization. As an example, though there are few projective manifolds $M$ endowed with holomorphic affine connections, since this implies that all Chern classes are zero, any such manifold is endowed with a finite map $f:M\longrightarrow \mathbb{P}^{N}$ for some integer $N\geq 1$. The canonical projective structure on $\mathbb{P}^N$ then induces through $f$ a meromorphic (flat) projective connection on $M$ (see \cite{BD}).

In this paper, we study the existence of \textit{meromorphic affine connections} (\autoref{affineconnection}) on complex compact surfaces of \textit{algebraic dimension} $a(M)=1$ (the algebraic dimension is defined above \eqref{algebraicreduction}). We almost obtain a classification of such pairs, that we call \textit{meromorphic affine complex compact surface of algebraic dimension one}, in the following sense. By the well-known work of Kodaira (\cite{KodairaCI},\cite{KodairaCII},\cite{KodairaSI}), complex compact surfaces of algebraic dimension one are known to be elliptic surfaces. Moreover, we can restrict ourselves to minimal surfaces since a meromorphic affine connection on a minimal surface with $a(M)=1$ is the \textit{pullback} (\autoref{pullback}) of a meromorphic affine connection on its minimal model (\autoref{minimalmodel}).  We first prove the following  (\autoref{invariantsconstant}):

\begin{theorem*} Any meromorphic affine complex compact surface of algebraic dimension one is an isotrivial surface.
\end{theorem*} 

Up to finite cover, a minimal meromorphic affine complex compact surface is thus a principal elliptic fiber bundle, and we have explicit descriptions of such surfaces in terms of their universal covers. The problem is then split in two : first we have to classify meromorphic affine principal elliptic fiber bundles, and then study the possible finite quotients of such pairs. This is completely done when the base curve is the projective line $\mathbb{P}^1$ (\autoref{classificationHopf1} and \autoref{quotientsHopf}) or an elliptic curve (\autoref{classificationtori},\autoref{quotientstori},\autoref{classificationprimaryKodaira} and \autoref{classificationsecondary}). In the remaining case (hyperelliptic curve), we describe a subset of codimension 3 in the space of meromorphic affine connections (\autoref{classificationgenus2}) extending the work by Klingler (\cite{Klingler}). However, we prove that there is no non-trivial quotient of such meromorphic affine surface (\autoref{quotientsprincipalellipticgenus2}). So there a no new examples arising from these principal elliptic bundles. These results can be compared to the result of \cite{InoueKobayashi} to obtain:

\begin{theorem}\label{nonew} Any meromorphic affine surface with $a(M)=1$ endowed with a meromorphic affine connection also admits a flat affine holomorphic connection.

\end{theorem}

As an example, no $K3$-surface with $a(M)=1$ admits a meromorphic affine connection.

\subsection{Organization of the paper}

The paper is organized as follows. In \autoref{section2}, we recall the notion of meromorphic affine connections. In \autoref{section3}, we collect classical facts from the work of Kodaira on elliptic surfaces that will be used in the rest of the paper, and prove \autoref{invariantsconstant} and \autoref{reductionprincipalelliptic}, reducing the problem of classification to the one of meromorphic affine principal elliptic bundles and their quotients, as explained above. Then, in \autoref{section4},\autoref{section5} we classify meromorphic affine complex compact surfaces of algebraic dimension one arising as quotients of principal elliptic fiber bundles over $\mathbb{P}^1$ or an elliptic curve. In \autoref{section6}, we treat the case of an hyperelliptic base curve. We give a description of a codimension three subset in the space of meromorphic affine connections on the corresponding principal elliptic bundle, in terms of meromorphic differential operators, and prove that no other examples arise from a quotient of such meromorphic affine surfaces.

\section{Meromorphic affine connections and minimality}\label{section2}
\subsection{Meromorphic connections and linearizations}
We begin by recalling the definitions of two objects appearing recurrently in this paper. Let $M$ be a complex manifold and $D=\underset{\alpha \in I}{\sum} D_\alpha$ an effective divisor. In the rest of the paper, we will denote by $TM$ the sheaf of holomorphic vector fields, $\Omega^1_M$ the sheaf of holomorphic one forms, and by: \begin{equation}\label{sheafmeromorphic} \mathcal{O}_M(*D) \end{equation} the sheaf of meromorphic functions on $M$ with poles a combination of the irreductible components $D_\alpha$.  If $\Psi:M\longrightarrow M'$ is a morphism of complex manifold $\Psi^*$ (resp. $\Psi_*$ stand for the pullback functor for sheaves (resp. the pushforward).

\begin{definition}\label{connection} Let $M$ be a complex manifold and $\mathcal{E}$ be a $\mathcal{O}_M$-module. A meromorphic connection on $M$ with poles at $D$ is a morphism of $\underline{\mathbb{C}}_M$-sheaves: $$\begin{array}{ccccc}\nabla&:& \mathcal{E}&\longrightarrow & \Omega^1_M\otimes \mathcal{E}(*D) \end{array}$$ satisfying the Leibniz identity: $$\forall f\in \mathcal{O}_M(U),\forall s\in \mathcal{E}(U),\, \nabla(fs)=df\otimes s + f\nabla(s)$$ 
\end{definition}

\begin{definition}\label{linearization} Let $G$ be a group and $M$ be a complex $G$-manifold with right (resp. left) action. Let $\mathcal{E}$ be a $\mathcal{O}_M$-module. A right (resp. left) $G$-linearization of $\mathcal{E}$ is a family $(\phi_g)_{g\in G}$ of isomorphisms: $$\begin{array}{ccccc}\phi_g&:& \mathcal{E}&\overset{\sim}{\longrightarrow}& g^* \mathcal{E}\end{array}$$ with the property: $$\begin{array}{cccc}\forall g,g'\in G,\, \phi_{gg'} &=& g'^*(\phi_g)\circ \phi_{g'} \\ &&&\\  \end{array}$$ A $G$-linearized $\mathcal{O}_M$-module is a pair $(\mathcal{E},(\phi_g)_{g\in G})$.
\end{definition}

In the case of a discrete group $G$, a $\mathcal{O}_M$-module with a $G$-linearization is a $G$-equivariant $\mathcal{O}_M$-module as defined in \cite{Lunts}. In this case, if $\overline{M}=G\backslash M$ is a complex manifold, then, denoting by $q:M\longrightarrow \overline{M}$ the quotient map, $G$ acts naturally on $q_* \mathcal{E}(U)$ for any $U\subset \overline{M}$. Hence, there is a functor from the category of $G$-linearized $\mathcal{O}_M$-modules to the one of $\mathcal{O}_{\overline{M}}$-modules mapping $(\mathcal{E},(\phi_g)_{g\in G})$ to the sheaf: \begin{equation}\overline{\mathcal{E}}=(q_* \mathcal{E}
)^G\end{equation} of $G$-invariant sections. It is an equivalence of categories in the case where the action of $G$ is free (\cite{Lunts}, Proposition 2.2.5). 

\begin{definition}\label{pullback} \begin{enumerate}\item Let $M$ and $M'$ be two complex manifold, $D'$ an effective divisor of $M'$ and $\nabla'$ be a meromorphic connection on a $\mathcal{O}_{M'}$-module $\mathcal{E}'$ with poles at $D'$.  Let $\mathcal{E}$ be a $\mathcal{O}_M$-module, $f:M\longrightarrow M'$ a isomorphism of complex manifold and $\varphi: \mathcal{E}\longrightarrow \mathcal{O}_M\otimes f^*\mathcal{E}'$ an isomorphism of $\mathcal{O}_M$-modules. The corresponding pullback $(f,\varphi)^\star \nabla'$ of $\nabla'$ is the meromorphic connection $\nabla$ on $\mathcal{E}$ with poles at $D=f^* D'$ defined by the commutative diagram: \begin{equation}\label{pullbackpsi2}\xymatrix{ \mathcal{E} \ar[rr]^{\nabla} \ar[d]_{\varphi} & &\Omega^1_{M} \otimes \mathcal{E}(*D) \\ \mathcal{O}_{M}\otimes f^* \mathcal{E}' \ar[rr]_{\overline{f^* \nabla'}} & &f^* \Omega^1_{M'} \otimes f^* \mathcal{E}'(*D')\ar[u]_{df^* \otimes \varphi^{-1}} }\end{equation} where: \begin{itemize}\item 
    $\begin{array}{ccccc} df&:& TM(*D) & \longrightarrow & \mathcal{O}_M\otimes f^* TM'(*D')  \end{array}$ is the sheaf-theoretic differential of $f$
    \item $\overline{f^* \nabla'}$ is the extension of the sheaf-theoretic pullback $f^* \nabla': f^* \mathcal{E}'\longrightarrow f^* \Omega^1_{M'}\otimes \mathcal{E}'(*D')$ by the Leibniz rule to $\mathcal{O}_M \otimes f^* \mathcal{E}'$. \end{itemize}
    \item Let $(\mathcal{E},(\phi_g)_{g\in G})$ be a $G$-linearized $\mathcal{O}_M$-module. A meromorphic connection $\nabla$ on $\mathcal{E}$ is invariant by $(\phi_g)_{g\in G}$ if $(g,\phi_g)^\star \nabla=\nabla$ for any $g\in G$. \end{enumerate}
    
\end{definition}

Since a holomorphic connection on a $\mathcal{O}_M$-module can be alternatively described as a global section of $\mathcal{H}om(\mathcal{E},J^1\mathcal{E})$ splitting the jet-sequence of $\mathcal{E}$ (see \cite{Atiyah}), the correspondance between linearized sheaves and sheaves on the quotient immediately implies the following:
\begin{lemma}\label{quotientconnection} Let $q:M\longrightarrow M'$ be a ramified cover between two complex manifolds, with Galois group $\Gamma \subset Aut(M)$. Let $\Phi_\Gamma=(\phi_\gamma)_{\gamma \in \Gamma}$ be a $\Gamma$-linearization of a locally free $\mathcal{O}_M$-module $\mathcal{E}$. Consider the locally free $\mathcal{O}_{M'}$-module  $\overline{\mathcal{E}}=(q_* \mathcal{E})^{\Phi_\Gamma}$. 
Suppose the existence of a holomorphic connection $\nabla$ on $\mathcal{E}$ which is invariant through any $\phi_\gamma$ ($\gamma \in \Gamma$). Then:
\begin{enumerate}\item If $q$ is unramified, then there exists induces a (unique) holomorphic connection $\nabla'$ on $\overline{\mathcal{E}}$ such that the pullback $(q,Id)^\star \nabla'$ (see \autoref{pullback}) coïncides with $\nabla$ when restricted to $q^* \overline{\mathcal{E}}$.
\item If $q$ is finite, then there exists a unique meromorphic connection $\nabla'$ on $\overline{\mathcal{E}}$, with poles supported on the ramification locus, such that the pullback $(q,Id)^\star \nabla'$ coïncides with the restriction of $\nabla$ to $q^* \overline{\mathcal{E}}$.
\end{enumerate}
\end{lemma}
\begin{proof}
\begin{enumerate}
    \item Consider the restriction $\nabla'$ of the morphism of $\underline{\mathbb{C}}_{M'}$-sheaves $q_* \nabla$ to the subsheaf $\overline{\mathcal{E}}$. It satisfies the Leibniz identity, and by definition of $(\gamma,\phi_\gamma)^\star \nabla=\nabla$ (\autoref{pullback}), the invariance of $\nabla$ by the linearization $\Phi_\Gamma$  implies that $\nabla'$ restricts as a morphism: $$\begin{array}{ccccc} \nabla'&:& \overline{\mathcal{E}} &\longrightarrow & (q_* \Omega^1_M\otimes \mathcal{E})^\Gamma \end{array}$$
    Now, we have a natural inclusion of sheaves: \begin{equation}\label{inclusionpushforward}\Omega^1_{M'}\otimes \overline{\mathcal{E}} =(q_* \Omega^1_M)^\Gamma \otimes (q_* \mathcal{E})^\Gamma\subset (q_* \Omega^1_M \otimes \mathcal{E})^\Gamma\end{equation} where the right handside is the subsheaf of invariant sections through the linearisation obtained by tensorizing the isomorphisms $(d\gamma^*)_{\gamma \in \Gamma}$ and $(\phi_\gamma)_{\gamma \in \Gamma}$. 
    Since $q$ is unramified, we have $$\mathcal{O}_M \underset{\mathcal{O}_{M'}}{\otimes} q^* \mathcal{O}_{M'} = \mathcal{O}_M$$ and since $\Gamma$ acts freely, we also have (see \cite{Lunts}, ): $$\mathcal{O}_M\otimes q^*\overline{\mathcal{E}}=\mathcal{E}$$ Hence the pullbacks of the vector bundles corresponding to the left and right handside of \eqref{inclusionpushforward} coïncide. Hence $\Omega^1_{M'}\otimes \overline{\mathcal{E}}=(q_* \Omega^1_M \otimes \mathcal{E})^\Gamma$, so that $\nabla'$ is in fact a holomorphic connection on $\overline{\mathcal{E}}$. Its pullback through $(q,Id)$ coïncides with the restriction of $\nabla$ to $q^* \overline{\mathcal{E}}$ by construction.
    \item The restriction of $q$ to the complement of its ramification locus $S\subset M$ is an unramified cover $q|_{M\setminus S}:M\setminus S \longrightarrow M'\setminus S'$. By $(1)$, we get a holomorphic connection $\nabla'$ on $\overline{\mathcal{E}}|_{M'\setminus S'}$ whose pullback through $(q|_{M\setminus S},Id)$ coïncides with the restriction of $\nabla|_{M\setminus S}$ to $q^* \overline{\mathcal{E}}|_{M\setminus S}$. 
    Denote by $j$ the inclusion of complex manifold of $M'\setminus S'$ in $M'$.  It remains to prove that $j_* \nabla'$ restricts to the subsheaf $\overline{\mathcal{E}}\subset j_* \overline{\mathcal{E}}|_{M'\setminus S'}=(j_*\mathcal{O}_{M'\setminus S'})\otimes \overline{\mathcal{E}}$ as a morphism with values in the subsheaf $\Omega^1_{M'}\otimes \overline{\mathcal{E}}(*S')$. It is sufficient to prove this property for the restriction of $(\overline{\mathcal{E}},j_*\nabla')$ to any complex submanifold $\Sigma\subset M$ of dimension one. Therefore, we assume from now on that $q$ is a finite cover between two complex curves $M,M'$. Pick $x\in S$ and $y=q(x)$. Since $\Gamma$ is finite, there exists a generator $\gamma\in \Gamma$ leaving a neigborhood $U$ of $x$ invariant, and local coordinates $z,z'$ at neighborhoods of $x$ and $y$ such that: $$q(z)=z^m \text{ and } \gamma \cdot z=e^{\frac{2i\pi}{m}}z$$ for some integer $m\geq 1$. Moreover, $\nabla$ is a flat connection on $\mathcal{E}$, so there is a local trivialization of $\mathcal{E}$ at a neighborhood $U$  of $x$ such that $\nabla$ identifies with the de Rham differential $d$ on $\mathcal{O}_U^{\oplus r}$ ($r$ stands for the rank of $\mathcal{E}$).  By the invariance of $\nabla$ through $\Phi_\Gamma$, the automorphism $\phi_\gamma$ maps the kernel $\underline{\mathbb{C}}_{U}^{\oplus r}$ of $d$ to its pullback $\gamma^* \underline{\mathbb{C}}_U^{\oplus r}$. It is therefore given by a constant linear transformation. Since $\gamma$ is of finite order, up to applying a constant linear change of trivialization, this linear transformation is diagonal with roots of the unity as eigenvalues, that is $\phi_\gamma$ corresponds to the isomorphism: $$\begin{array}{ccccc}\phi_\gamma &:& \mathcal{O}_U^{\oplus r}&\longrightarrow & \gamma^* \mathcal{O}_{M}^{\oplus r}\\&& (f_1(z),\ldots,f_r(z)) & \mapsto & (e^{\frac{2k_1 i\pi}{m}}f_1(e^{-\frac{2i\pi}{m}}z),\ldots,e^{\frac{2k_r i\pi}{m}}f_r(e^{-\frac{2i\pi}{m}}z)) \end{array}$$ with $k_j\in\{1,\ldots,m\}$. Hence, the restriction of $q^*\overline{\mathcal{E}}$ to $U$ is the subsheaf of $r$-uples of the form: $$s=(z^{m-k_1}g_1(z^m),\ldots,z^{m-k_r}g_r(z^m))$$ for some holomorphic functions $g_1,\ldots,g_r$ on $z^m(U)$. Then, through the trivialization of $\mathcal{E}|_U$ just defined: $$\nabla(s)=dz\otimes (\frac{m-k_1}{z}z^{m-k_1-1}g_1(z_1^m),\ldots,\frac{m-k_r}{z}z^{m-k_r-1}g_r(z^m))+ q^\star(dz')\otimes s' $$ for some $s'\in q^*\overline{\mathcal{E}}(U)$.  We can rewrite this section as: $$\nabla(s)= q^\star(dz') \otimes( \frac{z^{m-k_1}}{z^m} h_1(z),\ldots, \frac{z^{m-k_r}}{z^m} h_r(z))$$ for some meromorphic functions $h_1,\ldots,h_r$ on $z^m(U)$ with poles along $z^m(x)$. Hence $\nabla(s)$ is a section of $(\Omega^1_U)^\Gamma \otimes q^* \overline{\mathcal{E}}|_U(*x)$. This being true for any $x\in S$, and since $\Omega^1_{M'}=(q_* \Omega^1_M)^\Gamma$, we get: $$j_* \nabla'(\overline{\mathcal{E}})\subset \Omega^1_{M'}\otimes \overline{\mathcal{E}}(*S')$$ as required. This exactly means that $\nabla'$ defines a meromorphic connection on $\overline{\mathcal{E}}$, and this connection satisfies the pullback property by construction.
\end{enumerate}
\end{proof}

\subsection{Meromorphic affine connections and pullback}

We introduce the meromorphic geometric structure considered in this paper:
\begin{definition}\label{affineconnection} Let $M$ be a complex manifold and $D$ an effective divisor of $M$. A meromorphic affine connection on $(M,D)$ is a meromorphic connection on $TM$ with poles at $D$. \end{definition}

The pullback defined in \autoref{pullback} defines the category of meromorphic affine connections, with arrows given by the pullbacks through $(f,df)$ for $f$ an isomorphism of complex manfiolds.

\begin{lemma}\label{lemmapullback} Let $q:\hat{M}\longrightarrow M$ be a morphism of complex manifolds of the same dimension. Let  $\nabla$ be a meromorphic affine connection on $M$ and $\hat{\nabla}=q^\star \nabla$. Let $\Psi$ be an automorphism of $M$ and $\hat{\Psi}$ an automorphism of $\hat{M}$ lifting $\Psi$ through $q$. 

Then $\Psi^\star \nabla=\nabla$ if and only if $\hat{\Psi}^\star \hat{\nabla}=\hat{\nabla}$.

\end{lemma}
\begin{proof} Since $\hat{\Psi}$ is the lift of $\Psi$ through $q$, we have the following commutative diagram: 
 \begin{equation}\xymatrix{
T\hat{M} \ar[d]_{d\hat{\Psi}} \ar[r]^-{dq} & \mathcal{O}_{\hat{M}}\otimes q^* T\hat{M} \ar[d]^{\hat{\Psi}^*\otimes d\Psi} \\ \hat{\Psi}^* T\hat{M} \ar[r]_-{\hat{\Psi}^*dq}& \mathcal{O}_{\hat{M}}\otimes q^* T\hat{M} 
} \end{equation}
The equivalence asserted is then a direct consequence of the diagram defining a pullback (\autoref{pullback}).
\end{proof}

\subsection{Algebraic dimension and general property of elliptic surfaces}
Let $M$ be a compact complex manifold of complex dimension $n\geq 1$. Moishezon proved that the field of meromorphic functions $\mathbb{C}(M)$ is a field of finite transcendancy degree over the field of constant functions $\mathbb{C}$. This degree is called the \textit{algebraic dimension} of $M$ and denoted by $a(M)$. In particular $a(M)\leq n$, and there exists a bimeromorphic map $\Psi : M \longrightarrow M' $ and a holomorphic   map \begin{equation}\label{algebraicreduction} \pi : M' \longrightarrow N \end{equation} onto a complex compact manifold of dimension $a(M)$, with the property $\mathbb{C}(M)=\pi^* \mathbb{C}(N)$.

In this paper, we will focus on complex compact surfaces with $a(M)=1$.

A \textit{elliptic surface} is a holomorphic fibration $M\overset{\pi}{\longrightarrow} N$ of a complex compact surface over a (compact) complex smooth curve, such that for a generic $y\in N$, the fiber $M_y:=\pi^{-1}(y)$ is a (smooth) complex torus.

 We recall the following result from Kodaira (\cite{KodairaCI}):
 \begin{theorem}\label{a1elliptic} Any compact complex surface with $a(M)=1$ is the total space of an elliptic surface given by the algebraic reduction map \eqref{algebraicreduction}. \end{theorem}

 Moreover: \begin{theorem}\label{invariance} Let $M\overset{\pi}{\longrightarrow} N$ be an elliptic surface with $a(M)=1$. Any divisor $D$ of $M$ is of the form $D=\pi^* C$ for some divisor $C$ of $N$.

\end{theorem}
\begin{proof} See \cite{KodairaCI}, Theorem 4.3. \end{proof}
 
 Now, let $M\overset{\pi}{\longrightarrow}N$ be a general elliptic surface. The  fibers of $\pi$ which are not smooth elliptic curves are denoted by $M_{y_\beta}$ ($\beta\in J$), and for any local coordinate $z_\beta$ on $N$ centered at $y_\beta$, there exists an integer $m_\beta>1$ and an equation $f_\beta$ for $M_\beta$ with \begin{equation}\label{multiplicity} f_\beta = (z_\beta \circ \pi)^{m_\beta} \end{equation} The corresponding $y_\beta$ are the \textit{singular points} and the above integer will be called the \textit{multiplicty} of $y_\beta$ (resp. of the \textit{singular fiber} $M_\beta$), according to the work of Kodaira (\cite{KodairaCI}). 

\begin{proposition}\label{localformelliptic} Let $M\overset{\pi}{\longrightarrow}N$ be an elliptic surface,   $(y_\beta)_{\beta}$ the singular points and $N'$ their complement in $N$. Then: \begin{enumerate}
    \item For any $y\in N'$, there is a neighborhood $\overline{U}$ of $y$ in $N'$, and a holomorphic function $\tau: \overline{U}\longrightarrow \mathbb{H}$ such that : $$\begin{array}{ccc} \pi^{-1}(\overline{U}) & \simeq & \overline{U}\times \mathbb{C}/\langle \psi_1,\psi_2 \rangle \end{array} $$ where   $z_1,z_2$ are global coordinates adapted to the natural fibration and: \begin{equation}\label{psi12} \begin{array}{ccc} \psi_1(z_1,z_2)=(z_1,z_2+1) & \text{and} & \psi_2(z_1,z_2)=(z_1,z_2+\tau) \end{array} \end{equation}
    \item There exists a global holomorphic function $\tilde{\tau}:\tilde{N'}\longrightarrow \mathbb{H}$ on the universal cover of $\tilde{p}:\tilde{N'}\longrightarrow N'$, such that for any $y\in N'$, and any $\tau$ as in 1., $\tau=\tilde{\tau}\circ s$ for some section of $\tilde{p}$ near $y$.
\end{enumerate} \end{proposition}
\begin{proof}
    \begin{enumerate}
        \item Pick a simply connected neighborhood $\overline{U}$ of $y$ in $N'$. Then the fundamental group of $\pi^{-1}(\overline{U})$ is spanned by the image of any pair of generators $\gamma_1,\gamma_2$ of $\pi_1(M_y,x)$ ($x\in \pi^{-1}(y)$). By the Ehresmanh theorem, and since all smooth elliptic curves are diffeomorphic, shrinking $\overline{U}$ we can assume the existence of a diffeomorphism: \begin{equation}\label{topisolocal}
            \begin{array}{ccccc} \psi&:& \pi^{-1}(\overline{U}) & \longrightarrow & \overline{U}\times \mathbb{C}/\mathbb{Z}\oplus i \mathbb{Z} \end{array}
        \end{equation} such that $\pi=proj_1\circ \psi$. In particular, the cycles in $H_1(\overline{U}\times \mathbb{C}/\mathbb{Z}\oplus i \mathbb{Z},\mathbb{Z})$ corresponding to $1$ and $i$ are mapped to cycles $\gamma_1,\gamma_2\in H^1(\pi^{-1}(\overline{U}),\mathbb{Z})$, restricting as cycles $\gamma_{1,y},\gamma_{2,y}$ on any fiber $M_y=\pi^{-1}(y)$. These cycles form a basis of the real vector space $H^0(M_y,\Omega^1_{M_y})^*$, canonically identified with the universal cover $\mathbb{C}$ and its is well known that $M_y\simeq H^0(M_y,\Omega^1_{M_y})^*/\mathbb{Z}\gamma_{1,y}\oplus \mathbb{Z} \gamma_{2,y}$. We then let \begin{equation}\label{Psielliptic}\begin{array}{ccccc}\Psi&:&\pi^{-1}(\overline{U})& \overset{\sim}{\longrightarrow} &\overline{U}\times \mathbb{C}/\langle (z_1,z_2)\mapsto (z_1,z_2+k_1 \tau(z_1)+k_2) \rangle \\&& \psi^{-1}(z_1,[z_2]) & \mapsto & [[(z_1,z_2)]] \end{array}\end{equation} where $[[(z_1,z_2)]]$ is the class of $(z_1,z_2)\in \overline{U}\times \mathbb{C}$ in the target complex manifold. By construction, $\Psi$ lifts to the universal covers. The homology map induce by the lifting maps $(\gamma_{1,y},\gamma_{2,y})$ on $(1,\tau(y))$, so is $\mathbb{C}$-linear. By the remark preceding \eqref{Psielliptic}, $\Psi$ restricts on each fiber $M_y$ as an isomorphism onto $\mathbb{C}/\mathbb{Z}\oplus \tau(y)\mathbb{Z}$. Hence $\Psi$ is a biholomorphism.

        \item This is equivalent to the assertion that the sheaf $T$ whose local sections are the $\tau$ as in 1. contains a local system on $N'$. This is immediate since for any $y\in N'$, the set of $\tau\in \mathbb{H}$ such that $M_y \simeq \mathbb{C}/\mathbb{Z}\oplus \tau\mathbb{Z}$ is a finite set.
    \end{enumerate}
\end{proof}

Recall tat for any $\tau,\tau'\in \mathbb{H}$, the tori $\mathbb{C}/\mathbb{Z}\oplus \tau\mathbb{Z}$ and 
 $ \mathbb{C}/\mathbb{Z}\oplus \tau'\mathbb{Z}$ are isomorphic exactly when $\tau,\tau'$ lie in the same $SL_2(\mathbb{Z})$-orbit in $\mathbb{H}$, through the action: $$\begin{pmatrix} a & b\\ c & d\end{pmatrix}\tau:= \frac{a\tau+b}{c\tau +d}$$ Since the action is free, there exists an associated representation $\rho:\pi_1(N',y)\longrightarrow SL_2(\mathbb{Z})$ such that: \begin{equation}\label{rhoKodaira} \tilde{\tau}(\gamma\cdot \tilde{y})=\rho(\gamma)\cdot \tilde{\tau}(\tilde{y}) \end{equation} for any $\tilde{y}\in \tilde{p}^{-1}(y)$ and $\gamma \in \pi_1(N',y)$. 
 
 The associated local system over $N'$ is called the \textit{homological invariant} of $\pi$ by Kodaira. The following facts can be found in \cite{KodairaCII}:

 \begin{proposition}\label{coordinatesKodaira} Let $M\overset{\pi}{\longrightarrow}N$ be an elliptic surface, and $N'$ (resp. $M'$) the complement of its singular points $(y_\beta)_{\beta\in J}$ (resp. its singular fibers in $M$). Then: \begin{enumerate}
     \item If $J\neq\emptyset$ then $\tilde{N'}$ is an open subset of $\mathbb{C}$ and $H^1(\tilde{N'},\mathcal{O}_{\tilde{N'}})=\{0\}$. By \autoref{localformelliptic}, this implies $\tilde{M'}=\tilde{N'} \times \mathbb{C}\subset \mathbb{C}^2$ and we let $(z_1,z_2)$ be the canonical coordinates on $\mathbb{C}^2$.
     \item Let $y\in N'$. For any $\gamma\in \pi_1(N',y)$, the corresponding automorphism of $\tilde{N'}$ lifts canonically to an automorphism $\varphi_\gamma$ of $\tilde{M}'$. Moreover, $\varphi_\gamma$ commutes with $\psi_1$ and $\psi_2$ (see \eqref{psi12}), and the corresponding map $$\begin{array}{ccc}\pi_1(N',y) & \longrightarrow & \langle \psi_1,\psi_2, (\varphi_\gamma)_{\pi_1(N',y)}\rangle / \langle \psi_1,\psi_2\rangle \\ \gamma& \mapsto &\varphi_\gamma \mod \langle \psi_1,\psi_2\rangle \end{array}$$ is a homomorphism.
     \item In the case $J\neq \emptyset$, let $\gamma\in \pi_1(N',y)$ and $\varphi_\gamma$ the  automorphism of the universal cover of $M'$ as in 2.  There exists a constant $\mu_\gamma \in \mathbb{C}^*$ and a holomorphic function $f_\gamma$ on $\tilde{N'}$ such that: \begin{equation}\label{universalautomorphisms}
         \varphi_\gamma(z_1,z_2)=(\gamma\cdot z_1, \frac{\mu_\gamma}{c_\gamma \tilde{\tau}(z_1)+d_\gamma}z_2 + f_\gamma(z_1)) \end{equation} where
    $\rho(\gamma)=\begin{pmatrix}a_\gamma & b_\gamma\\c_\gamma & d_\gamma \end{pmatrix}$ (see \eqref{rhoKodaira}).
\item The pullback (see \autoref{pullback}) through the universal covering $\tilde{p}:\tilde{M'}\longrightarrow M'$ induces a bijection between the set of meromorphic affine connections on $M'$ and the set of meromorphic affine connections $\tilde{\nabla}$ on $\tilde{M'}$ satisfying: \begin{equation}\label{invarianceuniversal} \psi_1^\star \tilde{\nabla} = \psi_2^\star \tilde{\nabla} = \varphi_\gamma^\star \tilde{\nabla} = \tilde{\nabla} \end{equation} for any $\gamma \in \pi_1(N',y)$, where $\psi_1,\psi_2$ are defined as in \autoref{localformelliptic}.
 
 \end{enumerate}

 \end{proposition}
 \begin{proof}
     \begin{enumerate}
         \item In either case, $\tilde{N'}$ is isomorphic to $ \mathbb{P}^1,\mathbb{C}$ or $\mathbb{H}$. If $\tilde{N'}=\mathbb{P}^1$, then $N'=\tilde{N'}=\mathbb{P}^1$ since $N$ is compact, so that $J=\emptyset$. The converse is clearly true. Now if $J\neq \emptyset$, necessarly $\tilde{N'}$ is not $\mathbb{P}^1$, whence the assertion.
         \item Consider the complex manifold $\overline{M'}=\tilde{N'}\underset{N'}{\times}M'$ fitting in the following diagram: \begin{equation}\label{intermediatecovering} \xymatrix{ \overline{M'}\ar[r]^{\overline{p'}} \ar[d]_{\overline{\pi}} & M' \ar[d]^{\pi} \\ \tilde{N'} \ar[r]^{\overline{p'}} & N'  
         } \end{equation} where $\overline{p'}$ is the projection on the second factor, $\overline{\pi}$ the projection on the first one.  Then $\tilde{M'}$ is clearly isomorphic to the universal cover of $\overline{M'}$.  We denote by $[\tilde{y},x]$ the class of $(\tilde{y},x)\in \tilde{N'}\times M'$ in $\overline{M'}$.
         
         For any $\gamma \in \pi_1(N',y)$ denote by $\overline{\varphi}_\gamma$ the corresponding automorphism of $\overline{p'}$. Then: \begin{equation} \overline{\varphi}_\gamma([\tilde{y},x])=[\overline{\varphi}_\gamma(\tilde{y}),x]\end{equation} define an automorphism $\overline{\varphi}_\gamma$ of $\overline{M'}$ compaptible with $\overline{p}$ and $\overline{\pi}$. Then $\overline{\varphi}_\gamma$ lifts to an unique automorphism $\varphi_\gamma$ of the total space of the universal cover $\tilde{M'}$. The two remaining assertions follow from the fact that $\overline{M'}$ is the quotient of $\tilde{M'}$ by the subgroup $\langle \psi_1,\psi_2\rangle$ of automorphisms, and the map $\gamma \mapsto \overline{\varphi}_\gamma$ is clearly a homomorphism from the definition of $\overline{\varphi}_\gamma$.
         \item Let $(z_1,z_2)$ be coordinates on $\tilde{M'}=\tilde{N'}\times \mathbb{C}$, and $\tilde{\tau}$ as in \autoref{localformelliptic}, 2. Consider again the complex manifold $\overline{M'}$ as in the proof of 2. 
         By the proof of \autoref{localformelliptic}, 1., we have a canonical isomorphism of elliptic fibrations $$\overline{M'}\overset{\sim}{\longrightarrow}\tilde{N'}\times \mathbb{C}/\langle \psi_1,\psi_2\rangle$$ where $\psi_1,\psi_2$ are the automorphism defined by: \begin{equation}\begin{array}{ccc}\psi_1(z_1,z_2)=(z_1,z_2+1) & \text{and} & \psi_2(z_1,z_2)=(z_1,z_2+\tilde{\tau}(z_1)) \end{array}\end{equation} Let $\overline{\varphi}_\gamma$ be the lifting to $\overline{M'}$ of the automorphism of $\tilde{N'}$ corresponding to $\gamma\in \pi_1(N',y)$ (recall that $\varphi_\gamma$ is the lifting of $\overline{\varphi}_\gamma$ to the universal cover $\tilde{M'}$ of $\overline{M'}$). Let $\tilde{y}\in \tilde{N'}$ with $\overline{p'}(\tilde{y})=y$, and $z_1=z_1(\tilde{y})$. Then, by \eqref{rhoKodaira} and the above biholomorphism, we get that the multiplication by $\frac{1}{c_\gamma \tilde{\tau}(z_1)+d_\gamma}\in \mathbb{C}^\times$ induces a biholomorphism between the fibers $\overline{M'}_{\tilde{y}}$ and $\overline{M'}_{\gamma\cdot \tilde{y}}$.
         
         Now, recall that the automorphisms $\mathcal{A}(\overline{M'}_{\tilde{y}})$ are described by the exact sequence: $$\xymatrix{ 0\ar[r]& \overline{M'}_{\tilde{y}} \ar[r]&\mathcal{A}(\overline{M'}_{\tilde{y}}) \ar[r]& \mathbb{Z}/n_{\tilde{y}} \mathbb{Z} \ar[r]&0 }$$ where $\mathbb{Z}/n_{\tilde{y}}\mathbb{Z}$  corresponds to complex multiplications by a $n_{\tilde{y}}$-th root of the unity  inducing an involution on the elliptic curve ($n_{\tilde{y}}\leq 6$) and $\overline{M'}_{\tilde{y}}$ is identified as the subgroup of translations on itself.   

         Since $\overline{\varphi}_\gamma$ is an automorphism of the elliptic fibration $\overline{\pi}:\overline{M'}\longrightarrow \tilde{N'}$, the above remarks imply the existence of a holomorphic function $\mu_\gamma$ on $\tilde{N'}$, and a holomorphic section $\overline{f}_\gamma$ of $\overline{\pi}$ such that: $$\forall z_1\in \tilde{N'},\, \forall z_2\in \mathbb{C},\, \overline{\varphi}_\gamma([z_1,z_2])=[\gamma \cdot z_1, \frac{\mu_\gamma(z_1)}{c_\gamma\tilde{\tau}(z_1)+d_\gamma}z_2] + \overline{f}_\gamma(z_1)$$
         In particular, $\mu_\gamma$ is a constant, and since $\check{H}^1(\tilde{N'},\mathcal{O}_{\tilde{N'}})=\{0\}$, $\overline{f}_\gamma$ lifts to a section of $\tilde{\pi}:\tilde{M'}\longrightarrow \tilde{N'}$, that is a holomorphic function $f_\gamma$ on $\tilde{N'}$. Then $\varphi_\gamma$ is exactly the automorphism described in  the statement.
 
  \item If $\tilde{\nabla}$ is a meromorphic affine connection on $\tilde{M'}$ satisfying \eqref{invarianceuniversal}, then in particular it is invariant through the Galois group of the universal covering $\tilde{\overline{p}}:\tilde{M'}\longrightarrow \overline{M'}$ as in \eqref{intermediatecovering}. Thus, using \autoref{quotientconnection}, we have $\tilde{\nabla}=\tilde{\overline{p}}^\star \overline{\nabla}$ for some meromorphic affine connection on $\overline{M'}$. Since the automorphisms $\varphi_\gamma$ are the lifts of the elements $\overline{\varphi}_\gamma$ of the Galois group of the covering $\overline{p'}$, we also have $\overline{\nabla}=\overline{p'}^\star \nabla$ for some meromorphic affine connection on $M'$, that is $\tilde{\nabla}=\tilde{p}^\star \nabla$.

  Reciprocally, suppose that $\tilde{\nabla}=\tilde{p}^\star \nabla$ for some meromorphic affine connection $\nabla$ on $M'$. Then applying \autoref{lemmapullback} to the lifts $\psi_1,\psi_2$ and $\varphi_\gamma$ of the identity of $M'$ gives \eqref{invarianceuniversal}.
   \end{enumerate} 
  \end{proof}

 We will therefore use these well-known following facts about elliptic surfaces, due to Kodaira \cite{KodairaCI}. First, recall that given divisors $D_1,D_2$ on a complex compact surface $M$, there is a well defined  \textit{intersection number}: $$D_1\cdot D_2 := c_1(D_1)c_1(D_2)$$ where $c_1(D) \in H^1(M,\mathbb{Z})$ stands for the  first Chern class of the line bundle $\mathcal{O}_M(D)$. An \textit{exceptional curve} is then a rational smooth curve $C$ in $M$ such that $C\cdot C=-1$. 

 \begin{theorem}\label{theoremKodaira} Let $N$ be a smooth complex curve, $\mathcal{J}:\tilde{N}\longrightarrow SL_2(\mathbb{Z})\backslash \mathbb{H}$ and $\mathcal{G}$ a sheaf of subgroups of $SL_2(\mathbb{Z})$ as above. Then:
 \begin{enumerate}
     \item There exists a unique (up to biholomorphisms of elliptic surfaces) elliptic surface $B\overset{\pi_0}{\longrightarrow}N$ with invariants $\mathcal{J},\mathcal{G}$ and a global holomorphic section, called the \textit{basic member}.
     \item Any \underline{minimal} elliptic surface $M'\overset{\pi'}{\longrightarrow}N$ with invariants $\mathcal{J},\mathcal{G}$ and no multiple singular fiber is locally isomorphic to $B$.
   %  \item Let $M\overset{\pi}{\longrightarrow}N$ be an elliptic surface with invariants $\mathcal{J},\mathcal{G}$. Then there exists a $M'\overset{\pi'}{\longrightarrow}N$ as in 2. and a biholomorphism: $$\xymatrix{ M'\setminus S' \ar[r]^{L}_{\sim} \ar[d]_{\pi'}& M\setminus S \ar[d]^{\pi} \\ N\setminus \overline{S} \ar[r]_{Id}^{\sim} & N\setminus \overline{S}} $$ between the complement of the singular fibers (a \textit{logarithmic transformation}).
 
 \end{enumerate} \end{theorem}
 \begin{proof}
This immediately follows from Theorem 10.1 in \cite{KodairaCII}.

 \end{proof}

\subsection{Minimal model for meromorphic affine complex surface with algebraic dimension one}
By a well-known result of Grauert, if $C$ is such a curve, then there exists a complex compact surface $M_1$, and $x_1\in M_1$ such that $M$ is isomorphic to the \textit{blow-up} at $x_0$ of $M$:   $$\sigma: M\longrightarrow M_1$$ and $\sigma(C)=\{x_1\}$. In this case $a(M_1)=a(M)$, the restriction of $\sigma$ to $M\setminus C$ is an isomorphism onto $M_1\setminus\{x_1\}$ and $\sigma$ maps any fiber of the algebraic reduction of $M$ to a fiber of the algebraic reduction of $M_1$. Given a complex compact surface $M$, there is a finite number of exceptional curve, and thus composing the maps $\sigma$ obtained as above we get a map: $$\sigma_0: M\longrightarrow M_0$$ which restricts as an isomorphism between $M\setminus C_0$, where $C_0$ is the union of the exceptional curves, and $M_0\setminus \{x_0^1,\ldots,x_0^n\}$ where the $x_0^i$ are points. Again $a(M_0)=a(M)$ and $M_0$ will be called the \textit{minimal model} of $M$.

In particular: 

\begin{lemma}\label{minimalmodel} Let $M$ be a complex compact surface endowed with a meromorphic affine connection $\nabla$. Suppose that $M$ contains an exceptional curve and let $\sigma:M\longrightarrow M_0$ the minimal model of $M$. If $a(M)=1$, then there exists a meromorphic affine connection $\nabla_0$ on $M_0$ such that $\nabla=\sigma_0^\star \sigma_0$. \end{lemma}
\begin{proof} First, using the inverse of the restriction of $\sigma_0$ to $M\setminus C_0$, we obtain a meromorphic affine connection $\nabla_0$ on $M_0\setminus \{x_0^1,\ldots,x_0^n\}$ such that $\sigma_0^\star \nabla_0$ is the restriction of $\nabla$ to $M\setminus C_0$.

It remains to prove that $\nabla_0$ extend across the codimension two subset $\{x_0^1,\ldots,x^n_0\}$. For, pick $i\in\{1,\ldots,n\}$. Let $(u_1,u_2)$ be coordinates on a neighborhood $U_0$ of $x^i_0$ such that the intersection of any fiber of the algebraic reduction $\pi_0$ with $U_0$ is a fiber of $u_1$. Using these coordinates, the matrix of $\nabla_0$ in $(\der{}{u_1},\der{}{u_2})$ has the form $$du_1 \otimes \Gamma^k_{1j}+ du_2 \otimes \Gamma^k_{2j}$$ where $\Gamma_{ij}^k$ are meromorphic functions on $U_0\setminus \{x^i_0\}$.

Let $x'_0\in \sigma_0^{-1}(x_0)$ and let $\overline{U}$ be an open neighborhood of $\pi(x'_0)$ constructed as in point 1. of \autoref{coordinatesKodaira}. We let $(z_1,z_2)$ be the corresponding coordinates obtained using a local trivialisation of the covering $\overline{U}\times \mathbb{C}\longrightarrow \pi^{-1}(\overline{U})$ on a neighborhood $U'_0$ of $x'_0$. From the fact that $\sigma_0$ preserves the fibers of the algebraic reductions, it is clear that for any meromorphic $f$ function on an open subset of $U_0$, we have $\der{f}{u_2}=0$ if and only if $\der{f\circ \sigma_0}{z_2}=0$. Now, the pullback $\tilde{\nabla}$ of $\nabla$ to the universal cover $\overline{U}\times \mathbb{C}$ of $\pi^{-1}(\overline{U})$ has matrix $$dz_1\otimes (\Gamma')^k_{1j} + dz_2 \otimes (\Gamma')^k_{2j}$$ for some meromorphic functions $(\Gamma')_{ij}^k$. Moreover, by \autoref{invariance}, the poles of these functions are fibers of $z_1$. The invariance of $\tilde{\nabla}$ through $\psi_1,\psi_2$ thus implies that the restriction of the $(\Gamma')_{ij}^k$ to a generic fiber of $z_1$ is an elliptic holomorphic function on $\mathbb{C}$, that is a constant. By the previous remark we get $\der{\Gamma_{ij}^k}{u_2}=0$. In particular, each $\Gamma^k_{ij}$ extends across $x^i_0$ as a meromorphic function on $U_0$. Therefore, for any $i\in \{1,\ldots,n\}$, $\nabla_0$ extends as a meromorphic affine connection on a neighborhood $U_0$ of $x^i_0$, as required.

% Déplacer paragraphe a(M)=1 => elliptic + curves are fibers
% Explain locally, pi_0^{-1}(V) \simeq V\times C/... => poles are fibers => extension
\end{proof}

 \section{Reduction to the classification of principal elliptic surfaces}\label{section3}

By \autoref{minimalmodel}, we can restrict ourselves to the classification of minimal complex compact surface of algebraic dimension one endowed with meromorphic affine connections.
 
We now prove that the functional invariant of such a minimal elliptic surface is constant, hence this is a principal elliptic bundle up to a finite ramified covering.  This reduces the problem of the classification to the classification of meromorphic affine principal elliptic bundles and their quotients.

 \subsection{Reduction to isotrivial elliptic surfaces}
  Let $(M,D,\nabla)$ be a meromorphic affine complex compact surface with algebraic dimension one, and $M\overset{\pi}{\longrightarrow} N$ the corresponding elliptic surface. Our aim is to prove that this is an isotrivial surface, meaning the functional invariant $\tau$ is constant.
  
  Let $S$ (resp. $\overline{S}$) be the union of the singular fibers (resp. the singular points) in $M$ (resp. $N$). In view of \autoref{coordinatesKodaira},$(1)$, we can and do assume that $J\neq \emptyset$ or $N\neq \mathbb{P}^1$, otherwise $M$ is a Hopf surface, hence a principal elliptic fiber bundle over $N$. Let $M'=M\setminus S$ (resp. $N'=N\setminus \overline{S}$), $\tilde{M'}=\tilde{N'}\times \mathbb{C}\overset{\pi' = proj_1}{\longrightarrow} \tilde{N'}$  with adapted global coordinates $(z_1,z_2)$ as in \autoref{coordinatesKodaira}, $(1)$. This is the total space of the universal covering $p':\tilde{M'}\longrightarrow M'$.
  
  We let $\tilde{\nabla'}=p'^{\star}\nabla$, a meromorphic affine connection on $\tilde{M'}$ with poles $D'$ supported on fibers of $\pi'$ by \autoref{invariance}. Moreover, reemploying the notations from \autoref{coordinatesKodaira}, it is both invariant through the automorphisms: \begin{equation}\label{automorphismstori} \begin{array}{ccc}\psi_1(z_1,z_2)=(z_1,z_2+1) & \text{and} & \psi_2(z_1,z_2)=(z_1,z_2+\tau(z_1)) \end{array}\end{equation} and the automorphims $(\varphi_\gamma)_{\gamma \in \pi_1(N',y_0)}$ defined as in \eqref{universalautomorphisms}.

  Now we observe:
\begin{lemma}\label{lemmasemiperiodic} Let $g$ be a global holomorphic function on $\mathbb{C}$ with $g(z+1)=g(z)$ for any $z\in \mathbb{C}$. Suppose that there exists $\nu\in \mathbb{H}$ and $\mu \in \mathbb{C}$ such that \begin{equation}\label{tauaffine} \forall z\in \mathbb{C},\, g(z+\nu)=g(z)+\mu\end{equation} Then $g$ is constant and $\mu=0$. \end{lemma}
\begin{proof} The $1$-periodicity is equivalent to the existence of a holomorphic function on $\mathbb{C}\setminus \{0\}$ such that $g(z)=\overline{g}(e^{2i\pi z})$ for any $z\in \mathbb{C}$. 

Then the property \eqref{tauaffine} implies \begin{equation}\label{tauaffine2}\overline{g}(\lambda u)=\overline{g}(u)+\mu\end{equation} where $\lambda=e^{2i \nu}$ satisfies $0<|\lambda|<1$. Derivating this relation in $u$ implies $\overline{g}'(\lambda u) =\frac{\overline{g}'(u)}{\lambda}$. But $\overline{g}'$ is a Laurent series at $0$, with residue $0$. Since $|\lambda|<1$, $\lambda^{n-1}\neq \frac{1}{\lambda}$ except for $n=0$, and the identity \eqref{tauaffine2} implies $\overline{g}'=0$. Hence $\overline{g}$ is a constant, and the same holds for $g$. In particular $\mu=0$.
\end{proof}

\begin{corollary}\label{coefficientsinvariant} Let $(M,D,\nabla)$ be a meromorphic affine complex compact surface with algebraic dimension one, and let $(\tilde{M'},\tilde{D}',\tilde{\nabla}')$ be as above, with homological invariant $\tau$. Let  $(z_1,z_2)$ be adapted global coordinates as above and $(\der{}{z_1},\der{}{z_2})$ (resp. $(dz_1,dz_2)$) the corresponding trivialisation of $TM'$ (resp. $\Omega^1_{M'}$). Then either $\tau'=0$, or $\tilde{\nabla}'$ has matrix: \begin{equation}\label{matrixuniversalcomplement} \begin{array}{ccc}dz_1\otimes \begin{pmatrix} b(z_1)&0\\d(z_1)& c(z_1)\end{pmatrix} &+& dz_2\otimes \begin{pmatrix} 0&0\\a(z_1)&0 \end{pmatrix} \end{array} \end{equation} for some holomorphic functions $a,b,c,d$ on $z_1(\tilde{N'})\subset \mathbb{C}$.
\end{corollary}
\begin{proof} We assume $\tau'\neq 0$. Let $(f_{ij},g_{ij})_{i,j=1,2}$ be the meromorphic functions on $\tilde{M'}$ such that the matrix of $\tilde{\nabla}'$ in $(\der{}{z_1},\der{}{z_2})$ is: \begin{equation} \begin{array}{ccc}dz_1\otimes \begin{pmatrix} f_{11}&f_{12}\\f_{21}& f_{22}\end{pmatrix} &+& dz_2\otimes \begin{pmatrix} g_{11}&g_{12}\\g_{21}&g_{22} \end{pmatrix} \end{array} \end{equation} Recall that, given any automorphism $\psi$ of $\tilde{M'}$, the pullback $\psi^\star \tilde{\nabla}'$ is described by the following diagram:
\begin{equation}\label{pullbackpsi1}\xymatrix{ T\tilde{M'} \ar[rr]^{\psi^\star \tilde{\nabla}'} \ar[d]_{d\psi} & &\Omega^1_{\tilde{M'}}(\star \tilde{D})\otimes T\tilde{M'} \\ \psi^* T\tilde{M'} \ar[rr]_{\psi^* \tilde{\nabla}'} & &\psi^* \Omega^1_{\tilde{M'}}(\star \tilde{D}) \otimes \psi^* T\tilde{M'}\ar[u]_{d\psi^* \otimes d\psi^{-1}} }\end{equation} 
The differential of the automorphism $\psi_1$ from \eqref{automorphismstori} corresponds, through the trivialisation $(\der{}{z_1},\der{}{z_2})$, to the post-composition by $\psi_1$ of functions. In particular, by \autoref{invariance}, the invariance of  $\tilde{\nabla}'$ through $\psi_1$ implies that the restrictions of $(f_{ij},g_{ij})_{i,j=1,2}$ on a generic fiber of $\pi'$ are $1$-periodic holomorphic functions.
 
Now, using the diagram \eqref{pullbackpsi1}, the identity $\der{}{z_2}\lrcorner \psi_2^\star \tilde{\nabla}'=\der{}{z_2}\lrcorner \tilde{\nabla}'$ rewrites as: \begin{equation}\label{systeminvariance1}\left\{\begin{array}{lcl} g_{12}(z_1,z_2+\tau(z_1))&=& g_{12}(z_1,z_2) \\ g_{ii}(z_1,z_2+\tau(z_1))+(-1)^i \tau'(z_1)g_{12}(z_1,z_2+\tau(z_1))&=&g_{ii}(z_1,z_2) \\ g_{21}(z_1,z_2+\tau(z_1)) (z_1,z_2+\tau(z_1))&=&g_{21}(z_1,z_2) \\ + \tau'(z_1)(g_{11}-g_{22})(z_1,z_2+\tau(z_1)))\\ - (\tau'(z_1))^2g_{12}(z_1,z_2+\tau(z_1)) && \end{array} \right. \end{equation}

Since $\tau'\neq 0$,the first line implies that the restriction of $g_{12}$ to a generic fiber is a holomorphic elliptic function, that is a constant, i.e $g_{12}(z_1,z_2)=g_{12}(z_1)$.
Now, the second line and the previous fact show that the restriction of $g_{ii}$ to a generic fiber satisfies the conditions of \autoref{lemmasemiperiodic}, hence $g_{12}=0$ and $g_{ii}(z_1,z_2)=g_{ii}(z_1)$. This in turn implies, together with the third line, that the restriction of $g_{21}$ to a generic fiber satisfies the conditions in \autoref{lemmasemiperiodic}, so that $g_{11}(z_1)=g_{22}(z_1)$ and $g_{21}(z_1,z_2)=a(z_1)$.  

Now, we can rewrite similarly the system of functional equations corresponding to $\der{}{z_1}\lrcorner \psi_2^\star \nabla'=\der{}{z_1}\lrcorner \nabla'$, taking in account that $$d\psi_2^* (\psi_2^* \der{}{z_2})=\der{}{z_2}+\tau'(z_1)\der{}{z_1}$$ Since $g_{12}=0$, the first line will be indentical to the one of \eqref{systeminvariance1}, that is $f_{12}(z_1,z_2)=f_{12}(z_1)$. Then the second  line show that $f_{ii}$ satisfy conditions of \autoref{lemmasemiperiodic}, so that $f_{ii}(z_1,z_2)=f_{ii}(z_1)$ and $f_{12}-g_{11}=f_{12}+g_{22}=0$, while $g_{11}=g_{22}$ by the previous facts. Hence $g_{11}=g_{22}=0$. Finally, as before, the last line show that $f_{21}$ satisfy conditions of \autoref{lemmasemiperiodic}, i.e. $f_{21}(z_1,z_2)=f_{21}(z_1)$.

\end{proof}

\begin{theorem}\label{invariantsconstant}Any meromorphic affine complex compact surface of algebraic dimension one is an isotrivial elliptic surface. \end{theorem}
\begin{proof} We reemploy the above notations, and will describe explicitely the identity $$\der{}{z_1}\lrcorner \varphi_\gamma^\star \tilde{\nabla}'=\der{}{z_1}\lrcorner \tilde{\nabla}'$$ for any generator $\gamma$ of $\pi_1(N',y_0)$. 

First introduce the following notations. We let $g_\gamma$ be the function of $z_1$ corresponding to the matrix of the differential of the automorphism of $\tilde{N'}$ corresponding to $\gamma$. We also let: \begin{equation} \delta_\gamma(z_1)= \frac{\mu_\gamma}{c_\gamma \tau(z_1) + d_\gamma} \end{equation} where $\mu_\gamma,c_\gamma,d_\gamma$ are defined as in \autoref{coordinatesKodaira}.  

We first prove that the above identity implies $\delta'_\gamma(z_1)=0$. Indeed, the matrix of $d\varphi_\gamma$ in the basis $(\der{}{z_1},\der{}{z_2})$ and $(\varphi_\gamma^* \der{}{z_1},\varphi_\gamma^*\der{}{z_2})$ is: \begin{equation}\label{matrixdvarphi}C=  \begin{pmatrix}g_\gamma(z_1)&0\\z_2\delta'_\gamma(z_1)+f'_\gamma(z_1) &\delta_\gamma(z_1)\end{pmatrix} \end{equation}  In particular: \begin{equation} \der{}{z_1} \lrcorner C^{-1}dC = \begin{pmatrix} \frac{g'_\gamma(z_1)}{g_\gamma(z_1)}&0 \\ z_2(\frac{\delta''_\gamma(z_1)}{\delta(z_1)}-\frac{\delta'_\gamma(z_1)}{\delta(z_1)} \frac{g'_{\gamma}(z_1)}{g_\gamma(z_1)}) +h(z_1)&\frac{\delta'_\gamma(z_1)}{\delta(z_1)}\end{pmatrix} \end{equation} for some meromorphic function $h$ on $z_1(\tilde{N'})\subset \mathbb{C}$. Recalling the definition of the pullback (\autoref{pullback}), and focusing on the coefficient $c(z_1)$ of the matrix of $\tilde{\nabla}'$ (see \autoref{coefficientsinvariant}), the invariance by $\varphi_\gamma$ leads to: \begin{equation}  \begin{array}{ccc} z_2(\frac{\delta'_\gamma(z_1)}{\delta_\gamma(z_1)}-\frac{\delta'_\gamma(z_1)}{\delta_\gamma(z_1)}b(\gamma\cdot z_1) + \frac{\delta'_\gamma(z_1)}{\delta_\gamma(z_1)}c(\gamma \cdot z_1)) -z_2^2 (\frac{g'_\gamma(z_1)}{g_\gamma(z_1)}\delta''_\gamma(z_1)\frac{\delta'_\gamma(z_1)}{\delta_\gamma(z_1)}) + h(z_1)=c(z_1)  \end{array} \end{equation}  for some meromorphic function $h$ on $z_1(\tilde{N'})$. In particular: \begin{equation}\label{equationdelta} \begin{array}{cccc}\forall \gamma\in \pi_1(N',y_0),& \frac{g'_\gamma(z_1)}{g_\gamma(z_1)}\delta''_\gamma(z_1)\frac{\delta'_\gamma(z_1)}{\delta_\gamma(z_1)} &=&0 \end{array} \end{equation}

Now we fix a set $((\gamma_\alpha)_{\alpha \in I},(\gamma_{\epsilon,1},\ldots,\gamma_{\epsilon,2g})_{\epsilon=1,2})$ of generators for $\pi_1(N',y_0)$, where $\gamma_\alpha$ is obtained from a loop containing the singular point $y_\alpha$ in its bounded component, and $\gamma_1,\ldots,\gamma_{2g}$ span $\pi_1(N)$ ($g=g(N)$ is the genus). We prove that there are two possible cases:
\begin{enumerate}[label=\alph*)]
    \item $\tau$ is constant.
    \item For any $\alpha \in I$, $c_{\gamma_\alpha}=0$.
\end{enumerate}

Indeed, assume $\tau$ is not constant and pick $\alpha \in I$. Clearly, $c_{\gamma_\alpha}\neq 0$ implies that $\delta_{\gamma_\alpha}'' \delta'_{\gamma_\alpha} \neq 0$. Hence, in view of \eqref{equationdelta}, either $c_{\gamma_\alpha}=0$ or $g_{\gamma_\alpha}'=0$. Now, considering a suitable Möbius transformation $\phi$ in the connected component of $\mathbb{P}SL_2(\mathbb{R})$, and replacing the coordinate $z_1$ by $z'_1=\phi\circ z_1$, the equation \eqref{equationdelta} remains true when replacing $\tau$ by $\overline{\tau}=\tau\circ \phi^{-1}$, $g'_{\gamma_\alpha}$ 
 by the matrix of the differential of $\gamma_\alpha\cdot$ in the basis $\der{}{z'_1}$ which can be picked out non-zero, and keeping the same $c_{\gamma_\alpha}$ (since $\overline{\tau}$ have the same monodromy as $\tau$). Hence necessarly $c_{\gamma_\alpha}=0$.\\

 Now we prove that case b) can't happen in our situation. Indeed, it is known that there are generators $(\gamma_\alpha)_{\alpha\in I}$ and $(\gamma_{\epsilon,j})_{\substack{j=1,\ldots,g\\ \epsilon=1,2}}$ as above such that: \begin{equation}\label{fundamentalgroup} \underset{j=1}{\overset{g}{\prod}} [\gamma_{1,j},\gamma_{2,j}] = \underset{\alpha\in I}{\prod} \gamma_\alpha \end{equation} But property $b)$ means that the monodromy of $\tau$ is a representation $\rho$ with values in the abelian subgroup of translations in $SL_2(\mathbb{R})$. In particular, \eqref{fundamentalgroup} implies that the composition of the images $\rho(\gamma_\alpha)$ is trivial. 
 Moreover, following the proof of Theorem 7.3 in \cite{KodairaCII}, these are translations by $b_\alpha\geq 0$. From the previous remark, the sum of these positive integers is zero, so that $b_\alpha=0$, that is $A_{\gamma_\alpha}$ is the identity for any $\alpha \in I$.
 
  By Theorem 7.3 in \cite{KodairaCII}, this implies that the functional invariant $J$ have no pole on $N$. It is therefore a constant meromorphic function. As a consequence, $\tau$ is constant.

\end{proof}
 
 \subsection{Reduction to  elliptic fiber bundles}
 We now prove that the classification reduces to the one of meromorphic affine elliptic fiber bundles over complex curves, and their finite quotients.

From now, assume that $M\overset{\pi}{\longrightarrow} N$ is a \underline{minimal}  elliptic surface, of algebraic dimension one, with singular fibers $S=\pi^{-1}(\overline{S})$, endowed with a meromorphic affine connection $\nabla$ with poles $D$. By \autoref{invariantsconstant}, its invariants are constants. Since $N$ is a smooth compact complex curve, its is clear that there exists a finite   cover $\hat{N}\overset{\overline{q}}{\longrightarrow}N $, ramified at $\overline{S}$, such that the elliptic surface $\hat{M}\overset{\hat{\pi}}{\longrightarrow}\hat{N}$ obtained from the diagram: \begin{equation}\label{finitecover} \xymatrix{
\hat{M}=M\underset{N}{\times}\hat{N} \ar[r]^{q} \ar[d]_{\hat{\pi}} & M\ar[d]^{\pi} \\ \hat{N} \ar[r]_{\overline{q}} & N 
} \end{equation} where $q$ is the restriction of the first projection, and $\hat{\pi}$ is the restriction of the second projection, is an elliptic surface without multiple singular fiber.  Its invariants are respectively $\hat{J}=J\circ q$ and $q^* \mathcal{G}$, that is respectively a constant and a constant sheaf.

We get: \begin{theorem}\label{reductionprincipalelliptic} There is a surjective functor from the category of objects of the form $(\hat{M}\overset{\hat{\pi}}{\longrightarrow} \hat{N},\hat{\nabla},\Gamma)$ where: \begin{itemize}[label=$\bullet$]
\item $\hat{M}\overset{\hat{\pi}}{\longrightarrow}\hat{N}$ is a principal elliptic bundle of algebraic dimension one through its algebraic reduction.
\item $\hat{\nabla}$ is a meromorphic affine connection on $\hat{M}$, with pole $\hat{D}$,
\item $\Gamma$ is a finite group of automorphisms of the elliptic surface $\hat{M}\overset{\hat{\pi}}{\longrightarrow} \hat{N}$ and of the meromorphic affine connection $\hat{\nabla}$, with smooth quotient.
\end{itemize} to the category of \underline{minimal} meromorphic affine surfaces $(M\overset{\pi}{\longrightarrow}N,\nabla)$ with $a(M)=1$. This functor maps $(\hat{M}\overset{\hat{\pi}}{\longrightarrow}\hat{N},\hat{\nabla},\Gamma)$ to $(M,\nabla)$ where: \begin{itemize}\item $M=\Gamma\backslash \hat{M}$, $N=\overline{\Gamma}\backslash \hat{N}$ and $\overline{\Gamma}$ is the subgroup of finite automorphisms of $\hat{N}$ covered by an element of $\Gamma$, \item $\nabla$ is the meromorphic connection on $TM\ =(q_* T\hat{M} )^\Gamma$ with poles at $D=q(\hat{D})$, obtained by applying \autoref{quotientconnection} to the quotient map $q:\hat{M}\longrightarrow M$ corresponding to the action of $\Gamma$. \end{itemize}  Moreover, given any $(\hat{M},\hat{\nabla},\Gamma)$ in the first category, any $\gamma \in \Gamma$ lifts to the universal cover $\tilde{p}:\tilde{M}\longrightarrow \hat{M}$ as an automorphism: \begin{equation} \label{liftPsi}\tilde{\Psi}(z_1,z_2)=(\delta\cdot z_1, \mu z_2 + f_\delta(z_1) )\end{equation} where $\delta$ is the lift of an automorphism $\overline{\delta}\in Aut(\hat{N})$ to the universal covering $\tilde{N}$, $\mu\in \mathbb{C}^*$  and $f_\delta$ is a holomorphic function on $\tilde{N}$, and $(z_1,z_2)$ are coordinates on $\tilde{M}\subset \mathbb{C}^2$ as in \autoref{coordinatesKodaira}.  \end{theorem}
\begin{proof}
We begin by proving the equivalence of categories. Suppose that $(\hat{M}\overset{\hat{\pi}}{\longrightarrow}\hat{N},\hat{\nabla},\Gamma)$ is an object and $(M,\nabla)$ its image as in the statement. First, we prove that $M$ is of algebraic dimension one. Indeed, suppose that $f$ is a meromorphic function on $M$. Then $\hat{f}=f\circ q$ is an element of $\hat{\pi}^\# \mathbb{C}(\hat{N})$. By definition of $M\overset{\pi}{\longrightarrow} N$, $f$ is thus an element of $\pi^\# \mathbb{C}(N)$. Also, $M$ is a minimal surface. Indeed, if $C$ is an exceptional curve, then $q^* C=\hat{C}$ is a smooth rational curve in $\hat{M}$, contained in a fiber of $\hat{\pi}$, which can't be a principal elliptic fiber bundle. Hence, the functor is well-defined on objects and extends as a functor for the obvious choice of arrows (namely $\Gamma$-equivariant isomorphisms of meromorphic affine connections and isomorphisms of meromorphic affine connections).

Now, if $(M,\nabla)$ is an object of the target category, then we define $\hat{M}\overset{\hat{\pi}}{\longrightarrow}\hat{N}$ as in \eqref{finitecover} and $\hat{\nabla}=q^\star \nabla$. We prove that $\hat{M}\overset{\hat{\pi}}{\longrightarrow}\hat{N}$ is a principal elliptic bundle. First, recall that the invariants are respectively a constant for the functional invariant and a constant sheaf for the homological invariant. The basic member $\hat{B}$ associated to these invariants (see \autoref{theoremKodaira}) is $\hat{B}=\hat{N}\times \mathbb{C}/\Lambda$ for some lattice $\Lambda$. 
Moreover, $\hat{M}$ is a minimal surface. Indeed, suppose the existence of an exceptional curve $\hat{C}$ in $\hat{M}$. Since the algebraic dimension of $\hat{M}$ is one, the proof of Theorem 4.2 in \cite{KodairaCI} implies that $\hat{C}$ is a singular fiber of $\hat{\pi}$. But then its image $C$ through $q$ would also be an exceptional curve. Indeed, the restriction of $q$ to the support of $\hat{C}$ is a biholomorphism, so $C$ is a smooth rational curve. It is also a singular fiber of the minimal elliptic surface $M$. By the proof of Theorem 6.2 in \cite{KodairaCII}, we must have $C\cdot C\leq -1$. In one other hand, since $q$ is a finite cover, we also have $$-1 =\hat{C}\cdot \hat{C} = deg(\mathcal{N}_{\hat{C}}) =k C\cdot C$$ for some positive integer $k$, hence $C\cdot C=-1$ contradicting the minimality of $M$. From the point 2. of the \autoref{theoremKodaira}, $\hat{M}$ is locally isomorphic, as an elliptic fibration, to $\hat{B}$. Hence, $\hat{M}$ has no singular fiber. Up to considering a finite cover, $\hat{M}\overset{\hat{\pi}}{\longrightarrow}\hat{N}$ is therefore a principal elliptic fiber bundle. 
Finally, setting $\Gamma$ as the (finite) group of automorphisms of $q$, we get an object mapping to $(M,\nabla)$. 

Now let's prove the formula \eqref{liftPsi}. Let $\Psi$ be an automorphism of $\hat{\pi}$. It covers an automorphism $\overline{\delta}\in Aut(\hat{N})$ and we define $\delta\in Aut(\tilde{N})$ as its lift to the universal covering $\tilde{\overline{p}}:\tilde{N}\longrightarrow \hat{N}$. By construction, any lift $\tilde{\Psi}$ of $\Psi$ to the universal covering $\tilde{M}$ is an automorphism of $\mathbb{C}$-principal bundle covering $\delta$.
 
Consider the covering $\overline{q}:\overline{M}\longrightarrow M$, where $\overline{M}=\hat{M}\underset{\hat{N}}{\times}\tilde{N}$, obtained as in the proof of \autoref{coordinatesKodaira}. Any element of $\overline{M}$ is of the form $[\hat{y},x]$ for some $x\in \hat{M}$ and $\hat{y}$ in $\hat{N}$. Recall that $\tilde{M}$ is also the universal covering of $\overline{M}$.  Moreover, $\Psi$ lifts canonically to an automorphism $\overline{\Psi}$ of $\overline{M}$ defined by: \begin{equation}\label{barPsi} \overline{\Psi}([\hat{y},x])=[\overline{\delta}(\hat{y}),\Psi(x)] \end{equation} Then the formula \eqref{liftPsi} follows as  in the proof of point 3. of \autoref{coordinatesKodaira}.

\end{proof}

 \section{Principal elliptic surfaces over projective line}\label{section4}
We begin to apply the strategy proposed in the last section, with the case $\hat{N}=\mathbb{P}^1$. We first describe the compact complex surfaces which are principal elliptic bundles of algebraic dimension one on $\hat{N}$, namely the Hopf surfaces with $a(\hat{M})=1$. This description is in terms of the universal cover $\tilde{M}=\mathbb{C}^2\setminus \{0\}$.  Then we classify meromorphic affine connections on $\hat{M}$ (\autoref{classificationHopf1}). Finally we prove that a non trivial finite group of $Aut(\hat{\pi})$ has no fixed curve (\autoref{quotientsHopf}), any minimal meromorphic affine surface $M$ with $a(M)=1$ arising from a Hopf surface through the construction \eqref{finitecover} is again a Hopf surface.

 \subsection{Hopf surfaces of algebraic dimension one}
Recall the following characterization of Hopf surfaces among elliptic surfaces (see for example \cite{BarthPeters}):

 \begin{theorem}\label{characterizationHopf} Let $M$ be a complex compact surface. Then the following assertions are equivalent: \begin{enumerate}\item $M$ is a Hopf surface with algebraic dimension one
 \item The universal cover of $M$ is $\mathbb{C}^2\setminus \{0\}$ and $M$ is of algebraic dimension one
 \item There exists $\lambda \in \mathbb{C}^*$ with $|\lambda|^k \rightarrow 0$, an integer $d\geq 1$ such that the following diagram commutes:
 \begin{equation}\xymatrix{ \mathbb{C}^2\setminus \{0\}\ar[d]_{p_1} \ar[r]^{\rho_d} & M \ar[d]^{\pi} \\\mathbb{P}^1 \ar@{=}[r] & \mathbb{P}^1 }\end{equation} where $\rho_d$ is the quotient map corresponding to the action of $\Gamma_d=\langle (z_1,z_2)\mapsto (\lambda z_1,\lambda^{\frac{1}{d}}z_2)\rangle$ on $\mathbb{C}^2\setminus \{0\}$ and $p_1$ is the bundle map for the tautological bundle of $\mathbb{P}^1$.

 \end{enumerate}

 The Hopf surface corresponding to some fixed $\lambda$ and $d\geq 1$ will be denoted by $M_d$. 
 \end{theorem}
 
 As an example, the original Hopf surface is $M_{1}$ for $\lambda=\frac{1}{2}$.

 \subsection{Meromorphic affine elliptic bundles over the projective line}
Let $\hat{M}$ be a principal elliptic bundle over $\hat{N}=\mathbb{P}^1$, with algebraic dimension one. By \autoref{characterizationHopf}, there exists $d\geq 1$ such that $\hat{M}=M_d$ using the previous notation. In particular, the field of meromorphic functions on $\hat{M}$ is \begin{equation}\label{fieldHopf}  \tilde{p}^\#\mathbb{C}(\hat{M})= \mathbb{C}(\frac{z_1}{z_2^d}) \end{equation} where $\tilde{p}:\tilde{M}\longrightarrow \hat{M}$ is the universal cover and $(z_1,z_2)$ are the canonical coordinates of $\mathbb{C}^2$ restricted to $\tilde{M}=\mathbb{C}^2\setminus \{0\}$. From now on we identify $\mathbb{C}(\hat{M})$ with $\mathbb{C}(X)$ through the above identification. 
\begin{theorem}\label{classificationHopf1} Let $\lambda \in \mathbb{C}^*$ and $(M_d)_{d\geq 1}$ be the corresponding Hopf surfaces of algebraic dimension one (see \autoref{characterizationHopf}). Denote by $\mathcal{A}$ the affine space $\mathcal{A}$ of meromorphic affine connections on $\hat{M}$. Then: \begin{enumerate}\item There is an isomorphism of affine spaces: \begin{equation}\label{bijectionHopf1}\begin{array}{ccc}\mathbb{C}(X)^4 \times \mathbb{C}(X)^4 &\overset{\sim}{\longrightarrow} &\mathcal{A} \\ (P_{ij}(X),Q_{ij}(X))_{i,j=1,2} & \mapsto & \rho(\tilde{\nabla}_{P_{ij},Q_{ij}}) \end{array} \end{equation} where $\rho$ is the map sending any meromorphic affine connection on $\tilde{M}$ invariant by the Galois group of $\tilde{p}$ to the meromorphic affine connection on $\hat{M}$ constructed in \autoref{quotientconnection}, and: $\tilde{\nabla}_{P_{ij},Q_{ij}}$ is the meromorphic affine connection on $\tilde{M}$ whose matrix in $(\der{}{z_1},\der{}{z_2})$ is: \begin{equation}\label{matrixHopf}dz_1 \otimes \begin{pmatrix} P_{11}(\frac{z_1}{z_2^d})\frac{z_2^d}{z_1^2} & P_{12}(\frac{z_1}{z_2^d})\frac{z_2^{d-1}}{z_1} \\ P_{21}(\frac{z_1}{z_2^d})\frac{z_2}{z_1^2} & P_{22}(\frac{z_1}{z_2^d})\frac{z_2^d}{z_1^2} \end{pmatrix} +dz_2\otimes \begin{pmatrix} Q_{11}(\frac{z_1}{z_2^d})\frac{z_2^{d-1}}{z_1}& Q_{12}(\frac{z_1}{z_2^d})\frac{z_2^d}{z_1} \\ Q_{21}(\frac{z_1}{z_2^d})\frac{z_2^d}{z_1^2}&Q_{22}(\frac{z_1}{z_2^d})\frac{z_2^{d-1}}{z_1}\end{pmatrix}\end{equation}
    \item In particular, there exists non-flat meromorphic affine connections on any Hopf surface of algebraic dimension one, and exactly one holomorphic affine connection on any such manifold, that is the standard affine structure.
\end{enumerate}
\end{theorem}
\begin{proof}
Pick a meromorphic affine connection $\nabla$ on  $\hat{M}$ let $\tilde{\nabla}$ be its pullback on $\mathbb{C}^2\setminus \{0\}$ through $\tilde{p}$ (see \autoref{pullback}). Then $\tilde{\nabla}$ is by construction a meromorphic affine connection on $\mathbb{C}^2\setminus \{0\}$, which is $\Gamma_d$-invariant, and there are meromorphic functions $(f_{ij},g_{ij})_{i,j=1,2}$ on $\mathbb{C}^2\setminus \{0\}$ such that the matrix of $\tilde{\nabla}$ in $(\der{}{z_1},\der{}{z_2}) is$: \begin{equation}\label{matrixnablaHopf} \begin{array}{ccc}  dz_1 \otimes \begin{pmatrix}  f_{11} & f_{12} \\ f_{21} & f_{22} \end{pmatrix} &+& dz_2 \otimes \begin{pmatrix}  g_{11} & g_{12} \\ g_{21} & g_{22} \end{pmatrix} \end{array}\end{equation}  The invariance of $\tilde{\nabla}$ by the Galois group of $\tilde{p}$ is equivalent to the system of functional equations: \begin{equation}\label{invariancenabla}\left\{\begin{array}{cccc} \lambda& f_{ii}(\lambda z_1,\lambda^{\frac{1}{d}} z_2)&=&f_{ii}(z_1,z_2) \\\lambda^{\frac{1}{d}}&f_{12}(\lambda z_1,\lambda^{\frac{1}{d}}z_2) &=&f_{12}(z_1,z_2) \\ \lambda^{2-\frac{1}{d}}&f_{21}(\lambda z_1,\lambda^{\frac{1}{d}}z_2)&=&f_{21}(z_1,z_2) \\ \lambda^{\frac{1}{d}}&g_{ii}(\lambda z_1,\lambda^{\frac{1}{d}}z_2)&=&g_{ii}(z_1,z_2)\\ \lambda^{\frac{2}{d}-1}&g_{12}(\lambda z_1,\lambda^{\frac{1}{d}}z_2)&=& g_{12}(z_1,z_2) \\ \lambda &g_{21}(\lambda z_1,\lambda^{\frac{1}{d}} z_2)&=& g_{21}(z_1,z_2) 
\end{array}\right. \end{equation}

Now, it is straightforward that the functions $(f^0_{ij},g^0_{ij})_{i,j=1,2}$ given by setting $P_{ij}=Q_{ij}=1$ in \eqref{matrixHopf} define a solution of the above system. Given any other solution $(f_{ij},g_{ij})_{i,j=1,2})$, the quotients $h_{ij}=\frac{f_{ij}}{f^0_{ij}}$ and $k_{ij}=\frac{g_{ij}}{g^0_{ij}}$ are clearly elements of $q^\# \mathbb{C}(\hat{M})$. Hence $\nabla=\rho(\tilde{\nabla}_{P_{ij},Q_{ij}})$ for some $P_{ij},Q_{ij}\in \mathbb{C}(X)$. Reciprocally, given any elements $P_{ij},Q_{ij}\in \mathbb{C}(X)$, the meromorphic connection $\tilde{\nabla}_{P_{ij},Q_{ij}}$ on $\mathbb{C}^2\setminus \{0\}$ is invariant through the action of the Galois group of $\tilde{p}$, so by \autoref{quotientconnection} there exists a unique $\nabla$ on $\hat{M}$ with $\tilde{p}^\star \nabla = \tilde{\nabla}_{P_{ij},Q_{ij}}$.

         The curvature of the connection $\tilde{\nabla}=\tilde{\nabla}_{P_{ij},Q_{ij}}$  can be computed explicitely. As an example, for $Q_{ij}=0$, $P_{ij}=\delta_{ij}^{21}$ the matrix of the curvature $R_{\tilde{\nabla}}$  in the basis $\der{}{z_1},\der{}{z_2}$ is $$ dA+A\wedge A= dz_1\wedge dz_2 \otimes \begin{pmatrix} 0 & 0 \\\frac{1}{z_1^2} & 0\end{pmatrix} \neq 0$$ whence the assertion.

        However, when $\tilde{\nabla}_1$ is holomorphic, then $P_{ij}=Q_{ij}=0$ that is $\tilde{\nabla}$ is the standard affine structure on $\mathbb{C}^2$.
    
\end{proof}
\subsection{Quotients of meromorphic affine Hopf surfaces}

\begin{theorem}\label{quotientsHopf} Let $(M,\nabla)$ be a minimal meromorphic affine surface with $a(M)=1$ and suppose that the finite covering $\hat{M}$ from \eqref{finitecover} is a principal elliptic surface over $\mathbb{P}^1$. Then $\hat{M}=M$ and $(M,\nabla)$ is classified in \autoref{classificationHopf1}. \end{theorem}
\begin{proof} Let $\Gamma$ be the Galois group of $q:\hat{M}\longrightarrow M$ and $\overline{\Gamma}$ the Galois group of $\overline{q}:\hat{N}\longrightarrow N$. Suppose that $\overline{q}$ (and so $q$) admits a ramification point $y_\beta\in N$. By definition, this means that there exists on $M$ a multiple fiber $S_\beta=\pi^{-1}(y_\beta)$. Let $\hat{S}_\beta=q^{-1}(S_\beta)$. Then $\hat{S}_\beta$ is the curve obtained as the quotient of $\{z_1=0\}$ or $\{z_2=0\}$ in the univeral cover $\mathbb{C}^2\setminus \{0\}$ of the Hopf surface $\hat{M}=M_d$ (see \autoref{characterizationHopf}). These are precisely the inverse images of $0$ and $\infty$ through $\hat{\pi}:\hat{M}\longrightarrow\hat{N}=\mathbb{P}^1$. Without loss of generality (up to exchanging $z_1$ and $z_2$), we suppose that $\hat{S}_\beta$ is the quotient of $\{z_1=0\}$. This implies that $\overline{\Gamma}$ is a subgroup of $Aut(\mathbb{C})$. In particular, the action of any $\epsilon \in \Gamma$ lifts to the universal cover $\mathbb{C}^2\setminus \{0\}$ as an automorphism $\tilde{\epsilon}$ defined by: $$\tilde{\epsilon}(z_1,z_2)=(\mu(az_1+bz_2),\mu z_2)$$ for some $a,\mu \in \mathbb{C}^*$ and $b\in \mathbb{C}$. But then $\tilde{\epsilon}^m$ is an element of the Galois group of the universal cover $\rho_d : \mathbb{C}^2\setminus \{0\} \longrightarrow \hat{M}$ for some $m\geq 1$. Since this Galois group is spanned by $(z_1,z_2)\mapsto (\lambda z_1,\lambda^d z_2)$, this implies: $$\left\{ \begin{array}{ccc}\mu^m &=& \lambda^rd \\ &&\\ b\underset{k=1}{\overset{m}{\sum}} \mu^k&=& 0 \\ &&\\a \mu^m &=&\lambda^r \end{array}\right. $$ Since $|\lambda|>1$ we get $b=0$, so that $\tilde{\epsilon}(z_1,z_2)=(\lambda^{\frac{1}{r}}z_1,\lambda^{\frac{d}{r}} z_2)$ for some integer $r$. Now, by definition of $q$ and the remarks above, $\epsilon$ fixes the quotient of $\{z_1=0\}$. Hence $\lambda^{\frac{d}{r}}=\lambda^{ld}$ for some integer $l$, so that $\tilde{\epsilon}$ is in fact an element of the Galois group of $\rho_d$, i.e $\epsilon$ is the identity on $\hat{M}$. 

We have proved that either that $\overline{q}$ is an unramified finite cover. Hence $M$ has no multiple fiber, so it is a Hopf surface, in particular a principal elliptic bundle. This implies $\hat{M}=M$.  \end{proof}

\section{Principal elliptic surfaces over an elliptic curve and quotients}\label{section5}
In this section, following \autoref{reductionprincipalelliptic}, we classify meromorphic affine connections on (holomorphic) principal elliptic surfaces over a one torus $\hat{N}=\mathbb{C}/\Lambda'$, as well as their quotients.

Let $\hat{M}$ be a complex compact surface which is a holomorphic principal elliptic bundle over a torus. We first recall a result of Kodaira asserting that $\hat{M}$ corresponds to one of the two following examples: 
\begin{definition}\label{primary} A primary Kodaira surface over a torus $\hat{N}=\mathbb{C}/\Lambda'$  is an elliptic surface $\hat{M}\overset{\hat{\pi}}{\longrightarrow}\hat{N}$, where $\hat{M}=G\backslash \mathbb{C}^2$, with $\hat{\pi}(z_1,z_2)=[z_1]$ ($[z_1]$ stands for the class of $z_1$ in $\mathbb{C}/\Lambda'$) and the group $G\subset Aut(\mathbb{C}^2)$ spanned by $\psi_1,\psi_2$ as in \autoref{localformelliptic} (for some $\tau \in \mathbb{H}$) and the automorphisms $(\varphi_{\lambda'})_{\lambda'\in \Lambda'}$ defined by: \begin{equation} \varphi_{\lambda'}(z_1,z_2)=(z_1+\lambda',z_2+\overline{\lambda'}z_1+\beta_{\lambda'}) \end{equation} for some $\beta_{\lambda'}\in \mathbb{C}$. \end{definition}
\begin{definition}\label{torus} A two torus is an elliptic surface $\hat{M}\overset{\hat{\pi}}{\longrightarrow}\hat{N}$ where $\hat{M}$ is a quotient $G\backslash \mathbb{C}^2$, with $\hat{\pi}(z_1,z_2)=[z_1]$ and $G$ a subgroup of translations in $\mathbb{C}^2$. \end{definition}
\begin{theorem}\label{canonicaltrivial} Let $\hat{M}$ be a complex compact surface which is a holomorphic principal elliptic bundle over a torus. Then $\mathcal{K}_{\hat{M}}\simeq \mathcal{O}_{\hat{M}}$ and either: \begin{enumerate}[label=\alph*)]
    \item The first Betti number of $\hat{M}$ is odd if and only if $\hat{M}$ is a primary Kodaira surface.
    \item The first Betti number of $\hat{M}$ is even if and only if $\hat{M}$ is a two torus.
\end{enumerate}
\end{theorem}
\begin{proof} Let $(z_1,z_2)$ be coordinates on the universal cover $\tilde{M}$ of $\hat{M}$ as in \autoref{coordinatesKodaira}. Then the holomorphic volume form $dz_1\wedge dz_2$ is clearly invariant through the automorphisms $\psi_1,\psi_2$ and $(\varphi_\gamma)_{\gamma \in \pi_1(\hat{N},y)}$. Thus, it is the pullback of a global holomorphic volume form $\hat{\eta}$ on the covering $\overline{M}$ defined as in \eqref{intermediatecovering}, which is invariant through the Galois group of this covering. Hence, $\hat{\eta}$ is the pullback of a global holomorphic volume form on $M$, proving $\mathcal{K}_M\simeq \mathcal{O}_M$.

The second assertion is a part of a result of Kodaira (\cite{KodairaSI}, Chapter 6.), where we eliminated the K3 surfaces since these are elliptic surfaces over the projective line.
    
\end{proof}

 We are thus led to classify meromorphic affine primary Kodaira surfaces, meromorphic affine two tori, and their quotients.

\subsection{Meromorphic affine primary Kodaira surfaces}
Let $\hat{M}\overset{\pi}{\longrightarrow} \hat{N}= \mathbb{C}/\Lambda'$ be a primary Kodaira surface (\autoref{primary}) and $G$ the group such that $\hat{M}=G\backslash \mathbb{C}^2$. Suppose the existence of a meromorphic affine connection $\nabla$ on  $(M,D)$ for some divisor $D$. Since $\pi$ is a principal elliptic bundle, \autoref{invariance} implies $D=\pi^* C$ for some divisor on the one torus $N$.

Define $\mathcal{E}_0$ as the subspace of $\Lambda'$-elliptic functions, that is: \begin{equation}\label{E0}\begin{array}{ccc}\mathcal{E}_0 =\{h\in \mathcal{M}(\mathbb{C})\, |\, \forall \lambda'\in \Lambda',\, \delta_{\lambda'}(h)(z_1)=0\} &\text{where}& \delta_{\lambda'}(h)(z_1)=h(z_1+\lambda')-h(z_1)\end{array} \end{equation} Recall that $\mathcal{E}_0$ is the subfield of meromorphic functions obtained as the extension of $\mathbb{C}$ by two elements $\wp(z_1),\wp'(z_1)$, where: \begin{equation}\label{Weirestrass}\wp(z_1) = \frac{1}{z_1^2}+\underset{\lambda' \in \Lambda'\setminus \{0\}}{\sum} \frac{1}{(z_1-\lambda')^2}-\frac{1}{\lambda'^2} \end{equation} is Weirestrass elliptic function. Then define: \begin{equation}\begin{array}{ccc}\mathcal{E}_1 &=& \{ h \in \mathcal{M}(\mathbb{C})\, |\, \exists \chi_h \in Hom_\mathbb{Z}(\Lambda',\mathbb{C}),\, \forall \lambda'\in \Lambda',\, \delta_{\lambda'}(h)(z_1)=\chi_h(\lambda')\}\end{array}\end{equation} equipped with the natural linear map: \begin{equation}\begin{array}{ccccc}\chi &:& \mathcal{E}_1 & \longrightarrow & Hom_\mathbb{Z}(\Lambda',\mathbb{C}) \\&& h & \mapsto & \delta(h) \end{array}\end{equation} Clearly $h\in \mathcal{E}_1$ if and only if $h'\in \mathcal{E}_0$. In particular, $z_1 \in \mathcal{E}_1$, and the Weirestrass zeta function (a primitive of $\wp$) : \begin{equation} \zeta \in \mathcal{E}_1 \end{equation} This implies:

\begin{lemma}\label{E1} There is an exact sequence: \begin{equation} \xymatrix{ 0 \ar[r] & \mathcal{E}_0 \ar[r] & \mathcal{E}_1 \ar[r]^{\chi} & Hom_\mathbb{Z}(\Lambda',\mathbb{C}) \ar[r] & 0 }\end{equation} which splits through the linear map: \begin{equation}  \chi^{-1}(\alpha \chi_{\zeta} +\beta \chi_{z_1}) = \alpha \zeta(z_1)+\beta z_1  \end{equation} \end{lemma}
\begin{proof} The fact $ker(\chi)=\mathcal{E}_0$ is immediate by definition of $\chi$. By \eqref{Weirestrass}, $Res_0(\zeta)=-1$, while $Res_0(z_1)=Res_0(f)=0$ for any $f\in \mathcal{E}_0$. As a consequence, $\chi_\zeta=\chi(\zeta)$ and $\chi_{z_1}=\chi(z_1)$ are independant. Since $Hom_\mathbb{Z}(\Lambda',\mathbb{C})$ has dimension two, the above sequence is right-exact.

\end{proof}

Consider the pullback $\tilde{\nabla}=\tilde{\rho}^\star \nabla$, which is a meromorphic $G$-invariant affine connection on $\mathbb{C}^2$. By \autoref{invariance}, either $\nabla$ is flat or the pole $\tilde{D}$ of $\tilde{\nabla}$ is supported on a $\Lambda'$-invariant union of subvarieties $\{z_1=y_\alpha + \lambda'\}$. Suppose $\tilde{\nabla}$ is not flat. \\

There are meromorphic function $f_{ij},g_{ij}$ on $\mathbb{C}^2$, with poles supported at $\tilde{D}$, such that : \begin{equation}\label{matnabla} Mat_{\der{}{z_1},\der{}{z_2}}(\tilde{\nabla})=dz_1 \otimes \begin{pmatrix}f_{11} & f_{12}\\ f_{21} & f_{22} \end{pmatrix} + dz_2 \otimes \begin{pmatrix} g_{11} & g_{12} \\g_{21} & g_{22} \end{pmatrix}\end{equation} By the $\Lambda$-invariance of $\tilde{\nabla}$, the restriction of $f_{ij}$ and $g_{ij}$ to any fiber of $z_1$ is constant. That is $f_{ij},g_{ij}$ are elements of $(\pi\circ \tilde{\rho})^\# \mathbb{C}(N)$, and we will omit the second variable $z_2$ in the sequel.

Now, given any $\lambda'\in \Lambda'$, we have : \begin{equation}\label{pullbackformsKodaira} \begin{array}{ccc} \varphi^\star_{\lambda'}(dz_1)=dz_1  &\text{and} & \varphi^\star_{\lambda'}(dz_2)=dz_2+b_{\lambda'}dz_1 \end{array}\end{equation} where $\varphi_{\lambda'}$ are the elements of $G$ as in \autoref{primary}. Hence,  the invariance of $\tilde{\nabla}$ by $G$  rewrites as (see \cite{Vitter} p.238-239): \begin{equation}\label{invarianceKodaira}\forall \lambda'\in \Lambda',\, \left\{ \begin{array}{ccl}\delta_{\lambda'}(g_{12})(z_1)&=&0\\&&\\  \delta_{\lambda'}(g_{ii})(z_1)&=&(-1)^{i+1} \overline{\lambda'}g_{12}(z_1)\\&&\\ \delta_{\lambda'}(g_{21})(z_1 )&=& \overline{\lambda'}(g_{22}-g_{11})(z_1) -\overline{\lambda'}^2g_{12}(z_1) \\&& \\\delta_{\lambda'}(f_{12})(z_1)&=& -\overline{\lambda'}g_{12}(z_1)\\&&\\ \delta_{\lambda'}(f_{ii})(z_1) &=&  (-1)^{i+1} \overline{\lambda'}f_{12}(z_1)-\overline{\lambda'}g_{ii} (z_1)\\&&+ (-1)^i \overline{\lambda'}^2 g_{12}(z_1)\\&&\\ \delta_{\lambda'}(f_{21})(z_1) &=&   \overline{\lambda'} (f_{22}-f_{11}-g_{21})(z_1) \\&&  + \overline{\lambda'}^2(g_{22}(z_1)-g_{11}(z_1)-f_{12}(z_1)) + \overline{\lambda'}^3 g_{12}(z_1) \end{array}\right. \end{equation} Reciprocally, any family $(f_{ij},g_{ij})_{i,j=1,2}$ of meromorphic functions on $\mathbb{C}$ satisfying \eqref{invarianceKodaira} define a meromorphic affine connection on $M$.

Now, we  study the simultaneous solutions $(f_{ij},g_{ij})$ of the system of functional equations \eqref{invarianceKodaira}.

\begin{proposition}\label{matricesKodaira}  Let $\hat{\nabla}$ be a meromorphic connection on a primary Kodaira surface $\hat{M}$ as above. Let    $\alpha \zeta + \beta z_1$  the meromorphic function from \autoref{E1}. Then, using the notation $\mathcal{Z}(z_1)=\alpha \zeta(z_1)+\beta z_1$, the matrix of the meromorphic affine connection $\tilde{\nabla}=p^\star \hat{\nabla}$ in the basis $(\der{}{z_1},\der{}{z_2})$ is either:
\vspace{1\baselineskip}

\begin{enumerate}[label= \textbf{form \Roman*}:]
    \item 
\begin{equation}\label{matrix1Kodaira}\begin{array}{cc} &dz_1 \otimes \begin{pmatrix}[c|c] - (\mathcal{Z}^2 +\gamma_{11})g_{12}  & \mathcal{Z}g_{12} \\ \hline   \begin{array}{rr}&- (\mathcal{Z}^3 \\ + & c(\mathcal{Z}+k)^2 \\ +& d\mathcal{Z} + \gamma_{12})g_{12}  \end{array}& \begin{array}{ll}& ((\mathcal{Z} + \delta_{22}+\gamma_{12})^2 \\ +& \gamma_{22})g_{12} \end{array}\end{pmatrix}\\&\\ +& dz_2 \otimes \begin{pmatrix}  -(\mathcal{Z} + \delta_{11})g_{12} & g_{12} \\ -((\mathcal{Z}+(\delta_{22}-\delta_{11}))^2 + \delta_{21}) g_{12} &(\mathcal{Z} + \delta_{22})g_{12} \end{pmatrix} \end{array}\end{equation} with $g_{12}$ a non trivial $\Lambda'$-elliptic function, $\gamma_{ij},\delta_{ij}\in \mathcal{E}_0$, and $h=-\frac{2}{3}\delta_{11}$, $c=\delta_{11}+\delta_{22}+\gamma_{12}$ and $d,k$ satisfying \eqref{coeff21matrix1Kodaira}, or:
\item \begin{equation}\label{matrix2Kodaira}
\begin{array}{cc} &dz_1 \otimes \begin{pmatrix}[c|c] \begin{array}{ll} & - (g_{11}-f_{12})\mathcal{Z}\\ + &\gamma_{11} \end{array} & f_{12} \\ \hline  \begin{array}{ll} &(g_{11}-g_{22}+f_{12})\mathcal{Z}^2  \\ +&(\gamma_{11}-\gamma_{22}+\delta_{21})\mathcal{Z}\\+&\gamma_{21} \end{array}   &\begin{array}{ll}& - (g_{22}+f_{12})\mathcal{Z}\\+&\gamma_{22} \end{array}\end{pmatrix}\\&\\ +& dz_2 \otimes \begin{pmatrix}  g_{11} & 0 \\ - (g_{22}-g_{11})\mathcal{Z}+\delta_{21} & g_{22} \end{pmatrix} \end{array}\end{equation} where $g_{11},g_{22},f_{12},\gamma_{ij},\delta_{21}$ are arbitrary  $\Lambda'$-elliptic functions. 

\end{enumerate}
 
% Ancient form:
%\begin{equation}\label{matrix2Kodaira}\begin{array}{cc} &dz_1 \otimes \begin{pmatrix}[c|c] (\alpha \zeta + \beta z_1 + \gamma_{11})g_{11}+\delta_{11} & 0 \\ \hline \begin{array}{cl}&(\alpha \zeta + \beta z_1 +\gamma_{21})(\delta_{22}-\delta_{11}\\+&g_{11}(\gamma_{11}-\gamma_{22})-g_{21})+\delta_{21} \end{array} & (\alpha \zeta + \beta z_1 + \gamma_{22})g_{11}+\delta_{22} \end{pmatrix}\\&\\ +& dz_2 \otimes \begin{pmatrix} g_{11}& 0 \\ g_{21} & g_{11} \end{pmatrix} \end{array}\end{equation} with $g_{11},g_{21},f_{11},f_{22} ,\gamma_{ij},\delta_{ij}\in \mathcal{E}_0$ and $\alpha \zeta + \beta z_1$ is the meromorphic function from \autoref{E1}.
\vspace{1\baselineskip}

Reciprocally, any matrix as above is the matrix of the pullback $\tilde{\nabla}$ of some meromorphic affine connection $\hat{\nabla}$ on $\hat{M}$ through its universal covering $\tilde{p}:\tilde{M}\longrightarrow \hat{M}$. 
\end{proposition}
\begin{proof} Recall that the pullback of a meromorphic affine connection on $\hat{M}$ to its universal covering $\tilde{M}$ defines a bijection between meromorphic affine connections on $\hat{M}$ and  meromorphic functions $(f_{ij},g_{ij})_{i,j=1,2}$ on $\mathbb{C}$ solutions of \eqref{invarianceKodaira}.

\begin{itemize}[label=$\bullet$]
\item First suppose $g_{12}\neq 0$. In this case, the first line of \eqref{invarianceKodaira} is equivalent to $g_{12}\in \mathcal{E}_0\setminus \{0\}$, and applying \autoref{E1} to $\frac{g_{ii}}{g_{12}}$ shows that the second and fourth one is equivalent to $g_{ii}=(-1)^i (\alpha \zeta + \beta z_1+\delta_{ii})g_{12}$ and $f_{12}=(\alpha \zeta + \beta z_1+\gamma_{12})g_{12}$ for some elliptic functions $\delta_{11},\delta_{22},\gamma_{12}$. Rewriting the system \eqref{invarianceKodaira} in this case, we see that the third line is equivalent to: \begin{equation}\label{identitydelta} \forall \lambda'\in \Lambda',\, \delta_{\lambda'}(g_{21})(z_1)=\delta_{\lambda'}(-(\alpha \zeta + \beta z_1+(\delta_{22}-\delta_{11}))^2 g_{12})(z_1) \end{equation}   so that this line becomes equivalent to $$g_{21}=-((\alpha \zeta + \beta z_1+(\delta_{22}-\delta_{11}))^2 + \delta_{21}) g_{12}$$ for an arbitrary   elliptic function $\delta_{21}$. By the same principle, the fifth and sixth lines are now equivalent to $f_{ii}=(-1)^i ((\alpha \zeta + \beta z_1 + \delta_{ii}+\gamma_{12})^2 +\gamma_{ii})g_{12}$ and $f_{21}= -(\frac{1}{3}(\alpha \zeta + \beta z_1 + h)^3 + c(\alpha \zeta + \beta z_1 +k)^2 +\gamma_{21})g_{12}$ for an arbitrary elliptic functions $\gamma_{21}$ and with $h,c,k\in \mathcal{E}_0$ solutions of the system: 
\begin{equation}\label{coeff21matrix1Kodaira} \left \{\begin{array}{ccl} 3h^2 + 2ck &=& (\delta_{22}-\delta_{11})^2+(\delta_{22}+\gamma_{12})^2+(\delta_{11}+\gamma_{12})^2 + \delta_{21}+\gamma_{11}-\gamma_{22} \\&&\\ 2c+3h &=& 4(\delta_{22}+\gamma_{12}) \\&&\\3h+c&=& \delta_{22}-\delta_{11}+\gamma_{12} \end{array} \right. \end{equation}

We get the matrix \textbf{form I}.

\item Now suppose $g_{12}=0$. Then using the \autoref{E1} as before we get that the five first lines of \eqref{invarianceKodaira} are equivalent to $g_{12}=0$,$g_{11}, g_{22},f_{12}\in \mathcal{E}_0$ and $g_{21}=(\alpha \zeta+\beta z_1)(g_{11}-g_{22})+\delta_{21}$ and $f_{ii}=-(\alpha \zeta +\beta z_1)(g_{ii}+(-1)^i f_{12})+\gamma_{ii}$ for some arbitrary elliptic function $\gamma_{ii},\delta_{21}$.

Now,  the last line is equivalent to: $$\delta_{\lambda'}(f_{21})(z_1)= \delta_{\lambda'}((\alpha \zeta + \beta z_1 +\gamma_{21})^2 (g_{11}-g_{22}+f_{12})  +(\alpha \zeta +z_1)(\gamma_{11}-\gamma_{22}+\delta_{21}))$$ We get the \textbf{form II}.
\end{itemize}

\end{proof}

We obtained: \begin{theorem}\label{classificationprimaryKodaira} Let $\pi:M\longrightarrow \mathbb{C}/\Lambda'$ be a primary Kodaira surface, $\zeta,\alpha,\beta$ as above and $\tilde{\rho}:\mathbb{C}^2 \longrightarrow M$ its universal cover. Then: \begin{enumerate}
    \item The pullback of meromorphic connections through $\rho'$  gives a bijection between the set of meromorphic affine connections on $M$ and the set of meromorphic affine connections on $\mathbb{C}^2$ with matrix as in \autoref{matricesKodaira}.
    
    \item The only holomorpic affine connections on $M$ are the     $\nabla$ corresponding to $\tilde{\nabla}$ with matrix \eqref{matrix2Kodaira} with constant entries and $g_{21}=f_{22}-{11}$ (this was first proved by A. Vitter, see \cite{Vitter}, 5.b). In particular their curvature identically vanishes, and there are flat holomorphic affine structures on $M$.

    \item There exists non flat meromorphic affine connections on $M$.

\end{enumerate} \end{theorem}
\begin{proof}
\begin{enumerate}

    \item  By  the remark below \eqref{invarianceKodaira}, the set of meromorphic affine connections on $M$ is in bijection with the set of $\tilde{\nabla}$ with matrix form as in \autoref{matricesKodaira}.

    \item Among the matrix forms in \eqref{matricesKodaira}, the only possible form with holomorphic one forms as entries is \eqref{matrix2Kodaira}, with $f_{ij}\in \mathbb{C}$. In particular the curvature is identically zero, and picking $f_{11}=f_{22}=f_{21}=0$, we get that the standard holomorphic affine structure of $\mathbb{C}^2$ induces a holomorphic affine structure on $M$, thus recovering the result of Inoue,Kobayashi and Ochiai.

    \item In \eqref{matrix2Kodaira}, pick $f_{11}=\wp$, $f_{22}=f_{21}=0$. Then the curvature of $\tilde{\nabla}$ is $$R_{\tilde{\nabla}}=-\wp(z_1) dz_1 \wedge dz_2 \otimes dz_1 \otimes \der{}{z_2} \neq 0$$ We thus get a non flat meromorphic affine connection on $M$.

\end{enumerate}
    
\end{proof}

\subsection{Quotients: Meromorphic affine secondary Kodaira surfaces}

We now classify the quotients of meromorphic affine primary Kodaira surfaces.

The following fact, which comes from the proof of Theorem 39 in \cite{KodairaSII}, describe the possible quotients of a primary Kodaira surface $\hat{M}$:

\begin{lemma}\label{lemmasecondary} Let $M\overset{\pi}{\longrightarrow}N$ be a minimal elliptic surface with $a(M)=1$, endowed with a meromorphic affine connection, and $\hat{M}\overset{\hat{\pi}}{\longrightarrow} \hat{N}$ its finite ramified covering as in \autoref{reductionprincipalelliptic}. Suppose $\hat{M}$ has canonical trivial bundle. Then either $\hat{M}=M$ or $\mathcal{K}_M^{\otimes k}=\mathcal{O}_M$ for some $k\geq 2$. Moreover: \begin{enumerate}  \item$M=\Gamma\backslash\hat{M}$ where $\Gamma$ is a cyclic group acting freely and spanned by an automorphism $\tilde{\Psi}$ of the form \begin{equation}\label{PsisecondaryKodaira} \tilde{\Psi}(z_1,z_2)=(\nu z_1+\theta, \mu z_2  +az_1 +b)\end{equation}  where $\nu$ is a $k$-th root of the unity ($k\leq 6$), $\mu$ is a power of $\nu$ and $a,b\in \mathbb{C}$. \item Moreover, if $\hat{M}$ is a two torus and $\Gamma$ is not trivial, $\mu\neq \nu$ in \eqref{PsisecondaryKodaira}. \end{enumerate}
\end{lemma}
\begin{proof}\begin{enumerate}
    \item If $M$ is a principal elliptic bundle then $\hat{M}=M$ by construction of $\hat{M}$. We then suppose that $M$ is not a principal elliptic bundle and $\Gamma$ is not trivial. Moreover $kod(M)\leq kod(\hat{M})=0$. Since $M$ is minimal, by the Enriques-Kodaira classification (see \cite{BarthPeters}, Table 10 p.189), if $kod(M)=0$, then $\mathcal{K}_M^{\otimes k}=\mathcal{O}_M$ for some integer $k\geq 2$. Since $\mathcal{K}_{\hat{M}}$ is trivial, $q:\hat{M}\longrightarrow M$ is isomorphic to the unramified covering associated with $\mathcal{K}_M$, and the formula \eqref{PsisecondaryKodaira} follows from the proof of Theorem 38 in \cite{KodairaSII}. If $kod(M)=-\infty$,  then $M$ is a Hopf surface with $a(M)=1$, and therefore a principal elliptic bundle through its algebraic reduction. In particular, it has no singular fiber so that $\hat{M}=M$. This contradicts $\mathcal{K}_{\hat{M}}=0$ so necessarly $kod(M)=0$. In particular $\mathcal{K}_M^{\otimes k}$ is trivial for some $k\geq 1$ and the formula \eqref{PsisecondaryKodaira} can be recovered from Theorem 39 and Theorem 40 in \cite{KodairaSII}.

    \item Suppose $\mu=\nu$. By \autoref{lemmaautomorphismscovering}, $\tilde{\Psi}^k$ belongs to the subgroup spanned by the automorphisms $\Psi_1,\Psi_2$ and $(\varphi_{\lambda'})_{\lambda'\in \Lambda'}$, where $k$ is the order of $\Gamma$. Suppose $k>1$. The matrix of $d\tilde{\Psi}$ in the basis $(\der{}{z_1},\der{}{z_2})$ and $(\tilde{\Psi}^* \der{}{z_1},\tilde{\Psi}^*\der{}{z_2})$ is: $$\begin{pmatrix} \nu & 0 \\ a & \nu \end{pmatrix}$$, while the matrix of any element $\varphi$ in the subgroup spanned by $\Psi_1,\Psi_2$ and the $\varphi_{\lambda'}$ in $(\der{}{z_1},\der{}{z_2})$ and $(\varphi^* \der{}{z_1},\varphi^* \der{}{z_2})$ is the identity. 
    
    We get immediately $a=0$. Hence, $\tilde{\Psi}$ is the product of the automorphisms $z_1\mapsto \nu z_1 +\theta$ and $z_2 \mapsto \nu z_2+b$, with $\nu\neq 1$. Hence, there exists on $\tilde{M}=\mathbb{C}^2$ an isolated fixed point $\tilde{x}_0=(z_1^0,z_2^0)$. The coordinates $u_1=z_1-z_1^0$ and $u_2=z_2-z_2^0$ identify a neighborhood $\tilde{U}$ of $\tilde{x}_0$, invariant by $\tilde{\Psi}$, with $D(0,1)\times D(0,1)$. It conjugates the action of $\tilde{\Psi}$ and the action of the automorphism $\Psi_1$ of $D(0,1)\times D(0,1)$ defined by: $$\Psi_1(u_1,u_2)=(\nu u_1,\nu u_2)$$ For a suitably small $\tilde{U}$, $\tilde{p}$ restricts as a biholomorphism between $\tilde{U}$ and an open neighborhood $\hat{U}$ of $\hat{x}_0=\tilde{p}(\tilde{x}_0)$. Then $q(\hat{U})$ is an open neighborhood of some point $x_0\in M$, which is isomorphic to the analytic space obtained as the quotient of $D(0,1)\times D(0,1)$ by the subgroup spanned by $\Psi_1$ as above. It is clear since $\nu\neq 1$, that this space is not smooth, contradicting $(i)$ in \autoref{lemmaautomorphismscovering}. Hence $\mu=\nu$ implies $k=1$ that is $\Gamma$ is the trivial group.
\end{enumerate}

\end{proof}
 \begin{definition}\label{secondary} A secondary Kodaira surface is an elliptic surface $M\overset{\pi}{\longrightarrow}N$ which admits a primary Kodaira surface $\hat{M}\overset{\hat{\pi}}{\longrightarrow}\hat{N}$ as a finite unramified cover. \end{definition} 

 Hence, the classification of (non-trivial) quotients of meromorphic affine primary Kodaira surfaces is reduced to the classification of meromorphic affine secondary Kodaira surfaces.

The two following lemmas will be useful to simplify the invariance equations corrsponding to $\tilde{\Psi}^\star \tilde{\nabla}=\tilde{\nabla}$: 

 \begin{lemma}\label{lemmafunctionsinvolution} Let $\nu\in \mathbb{C}^\times \setminus \{1\}$ and $\theta \in \mathbb{C}$ such that $z_1 \mapsto \nu z_1 + \theta$ is an automorphism of the elliptic curve $\mathbb{C}/\Lambda'$, $r:\mathbb{C}/\Lambda' \longrightarrow \mathbb{P}^1$ the quotient by the subgroup spanned by this automorphism, and $\wp_0=Z_1\circ r$ where $Z_1$ is any primitive element in the field of meromorphic functions on $\mathbb{P}^1$. Then, for any integer $k\geq 0$ the set of $\Lambda'$-elliptic functions satisfying: \begin{equation} \label{propertyfunctioninvolution}f(\nu z_1+\theta)=\frac{1}{\nu^k}f(z_1)\end{equation} is $\mathbb{C}\left(\wp_0\right) \wp^{(k)}_0$.
 
 \end{lemma}
 
 \begin{proof}  Since $\mathbb{C}(\mathbb{P}^1)=\mathbb{C}(Z_1)$, we obviously have: $$r^\#\mathbb{C}(\mathbb{P}^1)=\mathbb{C}\left(\wp_0\right)$$ In one other hand, by definition of $r$, $r^\#\mathbb{C}(\mathbb{P}^1)$ is the subset of $\Lambda'$-elliptic functions invariant through the automorphism from the statement.

  In particular, derivating the invariance equation for $\wp_0$ fives : \begin{equation} \wp^{(k)}_0(\nu z_1 + \theta) = \frac{1}{\nu}\wp^{(k)}_0(z_1)  \end{equation} that is $\wp^{(k)}_0$ is a  $\Lambda'$-elliptic function satisfying \eqref{propertyfunctioninvolution}.  

  Now let $f$ be any function as in the statement. Then $g=\frac{f}{\wp^{(k)}_0}$ is an $\Lambda$-elliptic function which is invariant through the automorphism from the statement. Hence, $g\in \mathbb{C}\left(\wp_0\right)$ and finally $f\in \mathbb{C}\left( \wp_0\right) \wp^{(k)}_0$.

 \end{proof}

\begin{lemma}\label{lemmaindependancepowers} Let $\nu$ be a non trivial root of the unity and $\theta\in \mathbb{C}$ such that $\delta:z_1\mapsto \nu z_1 + \theta$ is a finite automorphism of $\mathbb{C}/\Lambda'$. Suppose $k\geq 1$ and $(h_{i,1},h_{i,0})_{i=0,\ldots,k}$ are elements of $\mathcal{E}_0$ such that: \begin{equation}\label{relationpowers} \underset{i=0}{\overset{k}{\sum}} h_{i,1}(z_1)(\alpha \zeta + \beta z_1)^i(\delta \cdot z_1)  + h_{i,0}(z_1) (\alpha \zeta + \beta z_1)^i(z_1)=0 \end{equation}
Then $h_{0,0}=0$ and $h_{i,0}=-\nu^i h_{i,1}$ for $i=1,\ldots,k$. Moreover, if there exists $i\in \{1,\ldots,k\}$ with $h_{i,0}\neq 0$ then $\theta=0$.
\end{lemma}
\begin{proof} We proceed by induction on $k\in \mathbb{N}_{\geq 1}$. 

Suppose the relation \eqref{relationpowers} holds for $k=1$. If $h_{1,1}=0$, then clearly $h_{1,1}=h_{1,0}=h_{0,0}=0$ since $(\alpha \zeta+\beta z_1)$ is not an element of $\mathcal{E}_0$. Thus, we can assume, without loss of generality, that $h_{1,1}=1$. Then derivating \eqref{relationpowers} gives : $$\nu(\alpha \wp + \beta)(\nu z_1+\theta) + \der{h_{1,0}}{z_1}(\alpha \zeta +\beta z_1)(z_1) + h_{1,0}(\alpha \wp + \beta) + h_{0,0}=0 $$  As before this implies $h_{1,0}=c\in \mathbb{C}$. But then, \eqref{relationpowers} becomes: $$(\alpha \zeta + \beta z_1)(\nu z_1 + \theta)+ c(\alpha \zeta + \beta z_1)(z_1) = -h_{0,0}$$ If $\theta\not \in \Lambda'$, then the left handside has residues summing to a nonzero value. This is impossible by the second Liouville's theorem. Hence $\theta\not \in \Lambda'$ implies $h_{1,1}=h_{1,0}=h_{0,0}=0$.

If $\theta\in \Lambda'$, then $(\alpha \zeta + \beta z_1)(\delta z_1)=\nu (\alpha \zeta + \beta z_1)(z_1)$. Again, since $(\alpha \zeta + \beta z_1)$ is not an element of $\mathcal{E}_0$, we get that \eqref{relationpowers} implies $h_{0,0}=0$ and $h_{1,0}=-\nu h_{1,1}$.

Suppose the lemma is true for $k\in \mathbb{N}_{\geq 1}$, and suppose \eqref{relationpowers} holds for $k+1$ in place of $k$. Applying the operator $\delta_{\lambda'}$ (see \eqref{E0}) to this relation gives: $$\forall \overline{\lambda'}\in \Lambda',\, \overline{\lambda'}^{k+1}(\nu^{k+1}h_{k+1,1}+h_{k+1,0}) + \overline{\lambda'}^k f_k + \ldots + \overline{\lambda'}f_1 + f_0 =0 $$ where $f_0,\ldots,f_k$ are $\mathbb{C}$-linear combinations of $h_{j,0}$, $h_{j,1}$ and $(\alpha \zeta + \beta z_1)^i( z_1)$  with $i,j\leq k$. In particular we get $\nu^{k+1}h_{k+1,1}+h_{k+1,0}=0$ and:   \begin{equation}\label{powersk} \underset{i=0}{\overset{k}{\sum}} h_{i,1}(z_1)(\alpha \zeta + \beta z_1)^i(\delta \cdot z_1)  + h_{i,0}(z_1) (\alpha \zeta + \beta z_1)^i(z_1)=0 \end{equation} By induction hypothesis, we get that $h_{i,0}=\nu^i h_{i,1}$ for $i=1,\ldots,k$. 

Finally, if there exists $i\in \{1,\ldots,k\}$ such that $h_{i,0}\neq 0$, then the induction hypothesis implies $\theta\in \Lambda'$. Also, by \eqref{powersk}, if $h_{k+1,0}\neq 0$ and $\theta\not \in \Lambda'$, then by the second Livouille's theorem $h_{k+1,1}(\alpha \zeta +\beta z_1)^{k+1}(\nu z_1+ \theta) + h_{k+1,0}(\alpha \zeta + \beta z_1)^{k+1}(z_1)$  have non trivial poles at the classes of $0$ and $\theta$ in $\mathbb{C}/\Lambda'$.

 \end{proof}
 
 \begin{theorem}\label{classificationsecondary} Let $M\overset{\pi}{\longrightarrow}N$ be a minimal meromorphic affine elliptic surface with $a(M)=1$, and suppose that the elliptic surface $\hat{M}\overset{\hat{\pi}}{\longrightarrow} \hat{N}$ from \eqref{finitecover} is a primary Kodaira surface. Denote by $p:\tilde{M}\longrightarrow \hat{M}$ the universal covering of $\hat{M}$ and $(z_1,z_2)$ coordinates as in \autoref{coordinatesKodaira}. Then: \begin{enumerate}
     \item\label{item1secondary} $M$ is the quotient of $\hat{M}$ by a cyclic group $\Gamma$, spanned by an element $\Psi$, which lifts to an automorphism $\tilde{\Psi}$ of $\tilde{M}$ of the form: \begin{equation}\label{tildePsisecondary} \tilde{\Psi}(z_1,z_2)=(\nu z_1 + \theta, \mu z_2 + bz_1 + c)\end{equation} where $z_1\mapsto \nu z_1 + \theta$ is an automorphism of the elliptic curve $\hat{N}=\mathbb{C}/\Lambda'$, $\mu$ is a power of $\nu$, and  $b,c\in \mathbb{C}$. Moreover $\nu=1$ if and only if $\hat{M}=M$.
     \item Suppose that $\nu\neq 1$ as in \autoref{item1secondary}  (i.e. $M$ is a secondary Kodaira surface). Let $\wp_0$ defined as in \autoref{lemmafunctionsinvolution}. Then the map $\nabla \mapsto \tilde{\nabla}=\tilde{p}^\star q^\star \nabla$ is a bijection between the set of meromorphic affine connections on $M$ and the set of meromorphic affine connections on $\tilde{M}$ with one of the following matrix forms in $(\der{}{z_1},\der{}{z_2})$: 
\vspace{1\baselineskip}
     
     \begin{enumerate}[label=\alph*) ]
         \item  \underline{if $\mu=\nu^2=1$ and $\theta=0$:}

        $$ dz_1 \otimes \begin{pmatrix} \gamma_{11}&0\\ (\alpha \zeta +\beta z_1)(\gamma_{11}-\gamma_{22}+\delta_{21})+\frac{a}{1-\nu}\delta_{21}+\gamma_{21} & \gamma_{22}  \end{pmatrix}  + dz_2 \otimes \begin{pmatrix}0 & 0 \\ \delta_{21} & 0 \end{pmatrix}$$ with $\gamma_{ii},\delta_{21} \in \mathbb{C}\left(  \wp_0\right)\wp'_0$ and $\gamma_{21}\in \mathbb{C}\left( \wp_0 \right) \wp''_0$.
        \vspace{1\baselineskip}
         
         \item \underline{if $\mu\neq 1$ or $\nu^2\neq 1$ or $\theta\neq 0$:}
         
         $$  dz_1 \otimes \begin{pmatrix} \gamma_{11}&0\\ \gamma_{21} & \gamma_{11} \end{pmatrix} $$ with $\gamma_{11}\in \mathbb{C}\left(  \wp_0\right)\wp'_0$ and $\gamma_{21}\in \mathbb{C}\left(\wp_0\right) \wp''_0$,

     \end{enumerate}
 \end{enumerate}
\vspace{1\baselineskip}

 In particular, there always exists meromorphic affine flat connections on any secondary Kodaira surfaces.

 \end{theorem}
 \begin{proof}\begin{enumerate}
     \item It is a consequence of \autoref{lemmasecondary}  and the definition of $\hat{M}$ in \autoref{reductionprincipalelliptic}. Indeed, $\nu =1$ implies that $\overline{q}:\hat{N}\longrightarrow N$ is an unramified covering, that is $M$ have no multiple singular fiber. By definition this implies $\hat{M}=M$.
     \item  Let $\nabla$ be a meromorphic affine connection on a secondary Kodaira surface $M\overset{\pi}{\longrightarrow}N$, and $\tilde{\nabla}$ the corresponding pullback to the universal covering $\tilde{M}$ of the primary Kodaira surface $\hat{M}$. Recall that the matrix of $\tilde{\nabla}$ in $(\der{}{z_1},\der{}{z_2})$ was described in \autoref{matricesKodaira}.

     Suppose that $\tilde{\nabla}$ has \textbf{form I} in \autoref{matricesKodaira} and let $\tilde{\Psi}$ be as in 1, so that $\nu \neq 1$. Then by \autoref{lemmaindependancepowers}, the equations corresponding to $\tilde{\Psi}^\star \tilde{\nabla}=\tilde{\nabla}$ imply: $$\left\{ \begin{array}{ccc} \nu (\alpha \zeta + \beta z_1)^2(\nu z_1+\theta)g_{12}(\nu z_1+\theta)&=&(\alpha \zeta + \beta z_1)^2(z_1) g_{12}(z_1)  \\ &&\\ (\alpha \zeta + \beta z_1)(\nu z_1 + \theta) g_{12}(\nu z_1 + \theta) &=& (\alpha \zeta + \beta z_1)(z_1) g_{12}(z_1) \\ && \\ \frac{\mu}{\nu}g_{12}(\nu z_1 + \theta)&=& g_{12}(z_1) \end{array}\right. $$ Since $g_{12}\neq 0$, using again \autoref{lemmaindependancepowers} we get: $$\left\{ \begin{array}{ccc} \frac{\nu^3}{\mu}&=&1 \\ && \\ \frac{\nu^2}{\mu}&=& 1 \end{array}\right.$$ so that $\nu=1$, that is $M$ is a primary Kodaira surface by 1. This contradicts our assumption on $M$.

     Hence $\tilde{\nabla}$ has \textbf{form II} in \autoref{matricesKodaira}. In this case, using the notation $\mathcal{Z}(z_1)=\alpha \zeta + \beta z_1$, the equations corresponding to $\tilde{\Psi}^\star \tilde{\nabla} = \tilde{\nabla}$ are: \begin{equation}\label{invariancesecondary2} \left\{ \begin{array}{lcl}    \mu g_{ii}(\nu z_1 + \theta)&=& g_{ii}(z_1)\\ && \\f_{12}(\nu z_1 + \theta)   &=& f_{12}(z_1)   \\ && \\   -\nu(g_{22}-g_{11})(z_1)(\mathcal{Z})(\nu z_1 + \theta) &=& (g_{22}-g_{11})(z_1)\mathcal{Z}(z_1)   \\ + \mu \nu \delta_{21}(\nu z_1 + \theta) &&+ \delta_{21}(z_1)\\&&\\ -(g_{ii}+(-1)^i \nu f_{12})(z_1)\mathcal{Z}(\nu z_1 + \theta)&=& -(g_{ii}+(-1)^i f_{12})(z_1) \mathcal{Z}(z_1) \\ + \nu \gamma_{ii}(\nu z_1 + \theta) +\frac{a}{\mu}g_{ii}(z_1)&& + \gamma_{ii}(z_1)\\&&\\ \nu (g_{11}-g_{22}+\nu f_{12})(z_1)\mathcal{Z}^2 (\nu z_1+\theta) &=& (g_{11}-g_{22}+f_{12})(z_1)\mathcal{Z}^2(z_1)\\ + -a\mathcal{Z}(g_{11}-g_{22})(\nu z_1 + \theta)  && \\+\nu^2(\gamma_{11}-\gamma_{22}+ \delta_{21})(\nu z_1+\theta)\mathcal{Z}(\nu z_1+\theta)&&   +(\gamma_{11}-\gamma_{22}+ \delta_{21})( z_1) \mathcal{Z}( z_1)\\ +\nu^2\gamma_{21}(\nu z_1 +\theta))+a \delta_{21}(\nu \zeta + \theta)&&+\gamma_{21}(z_1) \end{array}\right. \end{equation}

   Using \autoref{lemmaindependancepowers} we get the following restrictions. The third line of \eqref{invariancesecondary2} implies: \begin{equation}\label{involutiondelta21} \delta_{21}(\nu z_1 + \theta)=\frac{1}{\mu\nu}\delta_{21}(z_1) \end{equation} The fourth line implies: $$   (\nu-1)g_{ii}=(-1)^i(1-\nu^2)f_{12}  
   $$  In the same way, the fifth line implies: $$(\nu^2 -1)(g_{11}-g_{22})=(1-\nu^3)f_{12}$$ Since $\nu\neq 1$, we get: \begin{equation} g_{ii}=f_{12}=0\end{equation}
    The fourth line also implies: \begin{equation}\label{involutiongammaii}\gamma_{ii}(\nu z_1+\theta)=\frac{1}{\nu}\gamma_{ii}(z_1)-\frac{a}{\mu\nu}g_{ii}(z_1)=\frac{1}{\nu}\gamma_{ii}(z_1) \end{equation}

   Finally the fifth line of \eqref{invariancesecondary2} implies: \begin{equation}\label{involutiongamma21}\gamma_{21}(\nu z_1 +\theta)= \frac{1}{\nu^2}\gamma_{21}(z_1)-\frac{a}{\mu\nu^2}\delta_{21}(z_1) \end{equation} We distinguish between two cases: \begin{itemize}
       \item If $\mu\neq \nu$, then $\gamma=-\frac{a}{\nu^2-\mu\nu}\delta_{21}(z_1)$ is a solution of \eqref{involutiongamma21}. Hence, in view of \autoref{lemmafunctionsinvolution}, \eqref{involutiongamma21} is equivalent to: \begin{equation}
           \gamma_{21} \in -\frac{a}{\nu^2-\mu\nu}\delta_{21}(z_1) + \mathbb{C}(\wp_0)\wp''_0 
       \end{equation}
       \item If $\mu=\nu$, then either $\delta_{21}=0$ and \autoref{lemmafunctionsinvolution} shows that \eqref{involutiongamma21} is equivalent to $\gamma_{21}\in \mathbb{C}(\wp_0)\wp'_0$, or: $$\frac{\gamma_{21}}{\delta_{21}}(\nu z_1+\theta)=(\nu-\frac{1}{\nu^2})\frac{\gamma_{21}}{\delta_{21}}(z_1) - \frac{a}{\nu}$$ In the second subcase, since $z_1\mapsto \nu z_1+\theta$ has finite order $k$ with $\nu^k =1$, we get $(\nu-\frac{1}{v^2})^{k-1}+\ldots+(\nu-\frac{1}{\nu^2})+1=0$. Since moreover $k=2$ or $3$, we get a contradiction. Hence if $\mu=\nu$ then: \begin{equation} \begin{array}{ccc} \delta_{21}=0 & \text{and} & \gamma_{21}\in \mathbb{C}(\wp_0)\wp'_0 \end{array} \end{equation}
   \end{itemize}
   Moreover, the fifth line and \autoref{lemmaindependancepowers} imply: $$\nu(\gamma_{11}-\gamma_{22}+\frac{1}{\mu}\delta_{21})(z_1)=\frac{1}{\nu}(\gamma_{11}-\gamma_{22}+\frac{1}{\mu}\delta_{21})(z_1)$$  and $\theta\in \Lambda'$ whenever $\delta_{21}\neq 0$ or $\gamma_{11}\neq \gamma_{22}$. Comparing with \eqref{involutiondelta21} and \eqref{involutiongammaii}, the above equality implies: \begin{equation}
   \begin{array}{ccc} (\nu^2 = 1  \text{ and } \mu=1 \text{ and } \theta=0) & \text{or}& \gamma_{11}-\gamma_{22}=\delta_{21}=0 \end{array}
   \end{equation} 
We have proved that $\tilde{\nabla}$ has matrix form as described in the statement.

   Reciprocally, suppose that  $\tilde{\Psi}$  is the lift of the generator of $\Gamma$ as in 1.  Suppose also tat $\tilde{\nabla}$ is a meromorphic affine connection with matrix form as in the statement. Then \eqref{invariancesecondary2} is clearly satisfied, i.e.\  $\tilde{\nabla}$ is $\tilde{\Psi}$-invariant, and $M=\Gamma \backslash \hat{M}$ is a secondary Kodaira surface. This achieves the proof since the matrix form in case b) also appears in case a), and the torsion and curvature both vanish when $\gamma_{11}=\gamma_{22}=0$ in this case.

 \end{enumerate}
 
 \end{proof}

\subsection{Two tori}
In view of \autoref{canonicaltrivial}, to achieve the case $\hat{N}=\mathbb{C}/\Lambda'$ (see \autoref{reductionprincipalelliptic}), it remains to classify meromorphic affine two tori with $a(\hat{M})=1$ and their quotients.

Let $\hat{M}\overset{\hat{\pi}}{\longrightarrow}\mathbb{C}/\Lambda'$ be a two torus with $a(\hat{M})=1$, and $\mathcal{E}_0$ the subfield of $\Lambda'$-elliptic meromorphic functions. Let $\tilde{p}:\tilde{M}\longrightarrow \hat{M}$ be the universal covering, with coordinates $(z_1,z_2)$ as in \autoref{coordinatesKodaira}.  Then the automorphisms $(\varphi_{\lambda'})_{\lambda'\in\Lambda'}$ are translations in these coordinates.

Hence, for any meromorphic affine connection $\tilde{\nabla}$ on $\tilde{M}$,  with matrix: $$dz_1 \otimes (f_{ij})_{i,j=1,2} + dz_2 \otimes (g_{ij})_{i,j=1,2}$$ the condition $\varphi_{\lambda'}^\star \tilde{\nabla}=\tilde{\nabla}$ is equivalent to $f_{ij},g_{ij}\in \mathcal{E}_0$.

We immediately get:

\begin{theorem}\label{classificationtori} Let $\hat{M}$ be a two torus with $a(\hat{M})=1$ and $\tilde{p}:\tilde{M}\longrightarrow \hat{M}$ its universal covering with global coordinates $(z_1,z_2)$ as in \autoref{coordinatesKodaira}. Then the map from the set of meromorphic affine connections on $\hat{M}$ to $\mathcal{E}_0^8$ obtained by mapping $\hat{\nabla}$ to the coefficients $(f_{ij},g_{ij})$ of the matrix of $\tilde{\nabla}=\tilde{p}^\star \hat{\nabla}$ in $(\der{}{z_1},\der{}{z_2})$ is a bijection.
\end{theorem}

\subsection{Quotients of meromorphic affine two tori}

\begin{theorem}\label{quotientstori} Let $(M,\nabla)$ be a minimal meromorphic affine complex compact surface of algebraic dimension one, and suppose that the finite ramified covering $(\hat{M},\hat{\nabla})$ from \autoref{finitecover} is a meromorphic affine two torus (of algebraic dimension one).
Then either $M$ is a two torus, or $(M,\nabla)$ is a meromorphic affine secondary Kodaira surface, and such pairs were classified in \autoref{classificationsecondary}.
\end{theorem}
\begin{proof} If $M$ is not a two torus, then the finite covering $q:\hat{M}\longrightarrow M$ and therefore $\mathcal{K}_M$ is not trivial. The proof of point 1. in \autoref{lemmasecondary} only relies on the fact that $\mathcal{K}_{\hat{M}}$ is trivial, which is still satisfied in our situation. As a consequence, we get that the canonical global section $(dz_1\wedge dz_2)^{\otimes k}$ of $\mathcal{K}_{\tilde{M}}^{\otimes k}$, where $k$ is the order of $\nu$ and $\tilde{M}$ is the universal cover of $\hat{M}$, is invariant by the lift $\tilde{\Psi}$ of any automorphism of $q:\hat{M}\longrightarrow M$.  Hence $\mathcal{K}_{M}^{\otimes k}$ is trivial. By Theorem 38 of \cite{KodairaSII}, this implies that a finite unramified cover of $M$ is either a two torus or a primary Kodaira surface. 
In the second case, since $\mathcal{K}_M$ is not trivial, we immediately get that $M$ is a secondary Kodaira surface.
In the first case, we have $kod(M)=0$ and $a(M)=1$. By the Enriques-Kodaira classification (see \cite{BarthPeters}, Table 10 p.189), $M$ is either a secondary Kodaira surface, a two torus, or a K3 surface. Again, since $\mathcal{K}_M$ is not trivial, the only possiblity is a secondary Kodaira surface. 
\end{proof}

\section{Principal elliptic surface over an hyperbolic compact Riemann surface and quotients}\label{section6}

\subsection{Non-existence  on principal elliptic surfaces with $b_1(\hat{M})$ even}
Let $\hat{M}\longrightarrow \hat{N}$ be a principal elliptic surface over a Riemann surface $\hat{N}$ of genus $g\geq 2$, with $a(\hat{M}))=1$. Denote by $p:\mathbb{H}\times \mathbb{C}\longrightarrow \hat{M}$ its universal cover. From \autoref{coordinatesKodaira}, $p$ is the quotient by the automorphisms $\Psi_1,\Psi_2$ corresponding to a lattice associated with the fibers of $\hat{M}$, and by the automorphisms $\varphi_\gamma$ ($\gamma \in \pi_1(\hat{N},y)$) lifting the desk transformations of the universal cover $\overline{p}:\mathbb{H}\longrightarrow \hat{N}$. The later are of the form: $$\varphi_\gamma(z_1,z_2)=(\frac{a_\gamma z_1 + b_\gamma}{c_\gamma z_1 + d_\gamma},z_2 + f_\gamma(z_1))$$ in suitable global coordinates $z_1$ on $\mathbb{H}$ and $z_2$ on $\mathbb{C}$, with $f_\gamma$ a holomorphic function on $\mathbb{H}$. 

\begin{proposition}\label{b1even} Let $\hat{M}\longrightarrow \hat{N}$ be a principal elliptic surface as above. If $\hat{M}$ admits a meromorphic affine connection $\hat{\nabla}$, then $b_1(\hat{M})$ is odd.
    
\end{proposition}

\begin{proof} By \cite{KodairaCII}, Theorem 4.17 applied to the basic member $$B:=\mathbb{H}\times \mathbb{C}/\Lambda$$ in the same family as $\hat{M}$, we get that either $b_1(\hat{M})$ is odd or $\hat{M}$ is a deformation of $B$.

Suppose that $b_1(\hat{M})$ is even. In particular, there exists a diffeomorphism between the underlying smooth manifolds: \begin{equation}
    \xymatrix{ \hat{M}^\mathbb{R} \ar[r]^{\varphi} \ar[d] & B^\mathbb{R} \ar[d] \\ \hat{N}^\mathbb{R} \ar@{=}[r] & \hat{N}^\mathbb{R} } 
\end{equation} In particular $\varphi$ induces an isomorphism of $\pi_1(\hat{N},y)$-manifolds between the universal covers of $\hat{M}$ and $B$.

Pick $U\subset \hat{N}$ such that there exists a section of $p$ over $\hat{\pi}^{-1}(U)$, and thus a section of the universal cover of $B$ over the corresponding open subset $U\times V$ in $B$. Then, in the induced coordinates $(z_1,z_2)$ on $\hat{\pi}^{-1}(U)$ and $(z_1,u_2)$ on $U\times V$, $\varphi$ has the expression: \begin{equation}\label{varphiB}\varphi(z_1,z_2)=(z_1,z_2+f_U(z_2))\end{equation} for some $\mathcal{C}^\infty$ function $f_U$ on $U$. We fix $U$ and omit the subscript $U$ in the sequel.

Now, on $B$, we can consider a meromorphic affine connection $\nabla_0$ induced by the canonical holomorphic affine connection on $\mathbb{C}/\Lambda$ and any meromorphic affine connection on $\hat{N}$. Then we can consider the smooth connection $\overline{\nabla}_0:=\varphi^\star \nabla_0$ on the sheaf of complexified differentiable vector fields \begin{equation}\label{complexifiedtangent} T^\mathbb{C} (\hat{M}\setminus S):= T\hat{M}|_{\hat{M}\setminus S}\otimes \mathcal{C}^\infty_{\hat{M}\setminus S,\mathbb{C}} \end{equation} In the basis $\der{}{z_1},\der{}{\overline{z}_1},\der{}{z_2},\der{}{\overline{z}_2}$ induced by coordinates as above, the matrix of the pullback $\tilde{\nabla}_0:=p^\star \overline{\nabla}_0$ is: \begin{equation}\label{matrixnabla0} dz_1 \otimes \begin{pmatrix}a(z_1) & 0 \\ df \cdot a(z_1) + \der{}{z_1}df(z_1) & 0 \end{pmatrix} + d\overline{z}_1 \otimes \begin{pmatrix} 0 & 0 \\ \der{}{\overline{z}_1} df(z_1) & 0 \end{pmatrix} \end{equation} where $a$ is a meromorphic function on $\mathbb{H}$, identified with a $C^\infty$ function $a: \mathbb{H}\longrightarrow \mathfrak{g}\mathfrak{l}_2(\mathbb{R})$ and: $$df(z_1)=\frac{1}{2} \begin{pmatrix} \der{f}{z_1}& \der{\overline{f}}{\overline{z}_1} \\ \der{\overline{f}}{z_1} & \der{f}{\overline{z}_1} \end{pmatrix}$$ Suppose that $\hat{\nabla}$ is a meromorphic affine connection on $\hat{M}$, with poles at $D$, and let $S'=S+D$. Denote by $\hat{\nabla}^\mathbb{C}$ (resp. $\tilde{\nabla}^\mathbb{C}$) the unique extension of $\hat{\nabla}$ to a smooth connection on the sheaf \eqref{complexifiedtangent} (resp. its pullback to the universal cover $\mathbb{H}\times \mathbb{C}$). Then the matrix $A$ of $\tilde{\nabla}^\mathbb{C}$ in $\der{}{z_1},\der{}{\overline{z}_1},\der{}{z_2},\der{}{\overline{z}_2}$  satisfies: \begin{equation}\label{complexifiedmatrixnabla} \der{}{\overline{z}_1}\lrcorner A = \der{}{\overline{z}_2}\lrcorner A = 0 \end{equation} In one other hand, the difference $\tilde{\nabla}^\mathbb{C}-\tilde{\nabla}_0=p^\star(\hat{\nabla}^\mathbb{C}-\overline{\nabla}_0)$ is a $\Gamma$-equivariant section of the $\mathcal{C}^\infty_{\hat{M}\setminus S',\mathbb{C}}$-sheaf $T^\mathbb{C} (\hat{M}\setminus S')^* \otimes End(T^\mathbb{C} \hat{M}\setminus S)$. Let's decompose its matrix in the above basis as: \begin{equation}\label{matrixendo}\begin{array}{ccc} dz_1\otimes (f^{(1,0)}_{i,j})_{i,j=1,2} &+ &dz_2 \otimes (g^{(1,0)}_{i,j})_{i,j=1,2} \\ d\overline{z}_1\otimes (f^{(0,1)}_{i,j})_{i,j=1,2} &+ &d\overline{z}_2 \otimes (g^{(0,1)}_{i,j})_{i,j=1,2}) \end{array}\end{equation} where $f^{(p,q)}_{i,j}$ and $g^{(p,q)}_{i,j}$ are matrices with global sections of $\mathcal{C}^\infty_{\hat{M}\setminus S',\mathbb{C}}$ as entries.  Using \eqref{matrixnabla0} and \eqref{complexifiedmatrixnabla}, we get: \begin{equation}\label{equationnabladeform} f^{(0,1)}_{21} = - \der{}{\overline{z}_1} df(z_1) \end{equation} But $f$ is a $\Gamma$-invariant function on $\mathbb{H}$, so that, for any $\gamma \in \Gamma$: $$df(\gamma z_1)= \begin{pmatrix} \frac{1}{(c_\gamma z_1+d_\gamma)^2} & 0 \\ 0 & \frac{1}{\overline{(c_\gamma z_1+d_\gamma)}^2} \end{pmatrix} \cdot df(z_1)$$ In particular: $$\der{}{\overline{z}_1} df(\gamma z_1) =\begin{pmatrix} 0 & 0 \\ 0 & -\frac{2c_\gamma}{(c_\gamma z_1+d_\gamma)^3} \end{pmatrix}\cdot df(z_1) + \begin{pmatrix} \frac{1}{(c_\gamma z_1+d_\gamma)^2} & 0 \\ 0 & \frac{1}{\overline{(c_\gamma z_1+d_\gamma)}^2} \end{pmatrix} \cdot \der{}{\overline{z}_1} df(z_1)$$  Comparing with \eqref{equationnabladeform} and recalling the $\Gamma$-invariance of \eqref{matrixendo}, we get that for any $\gamma \in \Gamma$: $$-\frac{2}{(c_\gamma z_1+d_\gamma)^3} \der{f}{\overline{z}_1} = 0$$ Hence $f$ is a holomorphic function on $\mathbb{H}$, which precisely means that $\varphi$ is a biholomorphism. This contradicts the hypothesis $a(\hat{M})=1$.

Hence $\hat{M}$ does not admit any meromorphic affine connection.  \end{proof}
 
 \subsection{Principal elliptic surfaces with $b_1(\hat{M})$ odd}

 Let $\hat{M}\longrightarrow \hat{N}$ be a principal elliptic surface over a Riemann surface $\hat{N}$ of genus $g\geq 2$, with $a(\hat{M}))=1$ and  odd first Betti number. Denote by $p:\tilde{M}\longrightarrow \hat{M}$ and $\overline{p}:\mathbb{H}\longrightarrow \hat{N}$ the respective universal covers, and $(z_1,z_2)$ global coordinates on $\tilde{M}$ as in \autoref{coordinatesKodaira}. We will also use the notation $\Gamma=\pi_1(\hat{N},y_0)$ for the fundamental group of $\hat{N}$ at a fixed point.
 
 Then, up to finite unramified cover of the elliptic surface $\hat{M}$, for any $\gamma \in \pi_1(\hat{N},y_0)$, the automorphism $\varphi_\gamma$ from \autoref{universalautomorphisms} is of the form: \begin{equation}\label{varphihyperbolic} \varphi_\gamma(z_1,z_2)=(\frac{a\gamma z_1 + b_\gamma}{c_\gamma z_1 +d_\gamma}, z_2 + log_\gamma(c_\gamma z_1 +d_\gamma))\end{equation} for some $\begin{pmatrix} a_\gamma & b_\gamma \\ c_\gamma & d_\gamma \end{pmatrix}\in SL_2(\mathbb{R})$  and some determination $log_\gamma$ of the logarithm on $c_\gamma \mathbb{H}+d_\gamma$ (see \cite{Klingler}). 

 By \cite{Klingler}, there exists a holomorphic affine connection $\hat{\nabla}_0$ on $\hat{M}$. As in the previous sections, taking the matrix of the pullback $p^\star \hat{\nabla}-p^\star \hat{\nabla}_0$ in $(\der{}{z_1},\der{}{z_2})$  gives a bijection between the set of meromorphic affine $\hat{\nabla}$ connections on $\hat{M}$ and the set of solutions $(f_{ij},g_{ij})$ of the following system of functional equations, for any $\gamma \in \pi_1(\hat{N},y_0)$:

\begin{equation}\label{invariancehyperbolic1} \left\{\begin{array}{ccl} g_{12}(z_1) &=& (c_\gamma z_1 + d_\gamma)^2 g_{12}(\gamma z_1)\\ &&\\ g_{ii}(z_1)&=& g_{ii}(\gamma z_1) + (-1)^i c_\gamma (c_\gamma z_1+d_\gamma) g_{12}(\gamma\cdot z_1) \\&&\\ g_{21}(z_1)&=& \frac{1}{(c_\gamma z_1 + d_\gamma)^2}g_{21}(\gamma z_1) + \frac{c_\gamma}{c_\gamma z_1+d_\gamma} (g_{22}-g_{11})(\gamma z_1) - c^2_\gamma g_{12}(\gamma z_1)\end{array} \right. \end{equation} 
and: 

\begin{equation}\label{invariancehyperbolic2} \left\{\begin{array}{ccl}  f_{12}(z_1)&=& f_{12}(\gamma \cdot z_1) + \frac{c_\gamma}{c_\gamma z_1+d\gamma} g_{12}(\gamma z_1) \\ &&\\ f_{ii}(z_1)&=& \frac{1}{(c_\gamma z_1+d_\gamma)^2} f_{ii}(\gamma z_1) + (-1)^i \frac{c_\gamma}{c_\gamma z_1+d_\gamma} f_{12}(\gamma z_1) + \frac{c_\gamma}{c_\gamma z_1+d_\gamma} g_{ii}(\gamma z_1) \\&&\\f_{21}(z_1)&=& \frac{1}{(c_\gamma z_1+d_\gamma)^4} f_{21}(\gamma z_1) + \frac{c_\gamma}{(c_\gamma z_1+d_\gamma)^3}(f_{22}-f_{11})(\gamma z_1) \\ &&- \frac{c_\gamma^2}{(c_\gamma z_1+d_\gamma)^2} f_{12}(\gamma z_1) + \frac{c_\gamma}{c_\gamma z_1+d_\gamma}g_{21}(\gamma z_1)
\end{array} \right. \end{equation}

We will  describe, in terms of certain differential operators on a line bundle, a codimension three subset of this solutions, namely the one satisfying: \begin{equation}\label{codim2sol}g_{12}=g_{21}=f_{12}+g_{11}-g_{22}=0  \end{equation}  We define $\mathcal{A}_{\hat{M}}^+$ as the (codimension three) affine subspace of meromorphic affine connections on $\hat{M}$ satisfying the above condition.

We begin by preliminaries facts on differential operators on line bundles. These facts will be applied on a fixed line bundle on $\hat{N}$ constructed as follow. Consider the $\Gamma$-linearization $(\alpha_\gamma)_{\gamma \in \Gamma}$ on the trivial module $\mathcal{O}_{\mathbb{H}}$ (see \autoref{linearization}) given by: \begin{equation}\label{linearizationL} \begin{array}{cccccc}\alpha_{\gamma}&:& \mathcal{O}_\mathbb{H} & \longrightarrow & \gamma^* \mathcal{O}_\mathbb{H} \\ && f & \mapsto& (c_\gamma z_1 + d_\gamma)f\circ \gamma^{-1} \end{array}\end{equation} and define $\mathcal{L}$ to be the corresponding line bundle on $\hat{N}$ through the equivalence between linearized modules and modules on the quotient: \begin{equation}\label{bundleL} \mathcal{L}= (\overline{p}_* \mathcal{O}_{\mathbb{H}})^{\alpha} \end{equation}

Given a locally fre sheaf $\mathcal{E}$ of $\mathcal{O}_M$-modules on a  manifold, we can consider the sheaf $J^r \mathcal{E}$ (we refer to \cite{Atiyah} and \cite{GualtieriPym} Definition 2.21 for a definition).  These sheaves fit in exact sequences for $r\geq 1$: \begin{equation}\label{jetsequence}\xymatrix{ 0 \ar[r] & (\Omega^1)^{\otimes r}\otimes \mathcal{E} \ar[r] & J^r \mathcal{E} \ar[r]^{\pi^{r-1}} & J^{r-1}\mathcal{E} \ar[r]& 0}
\end{equation} where $\pi^{r-1}$ stands for the truncation map. These maps generalize by compositions to maps: \begin{equation}\label{truncation}\begin{array}{ccccc}\pi_r^s&:& J^r\mathcal{E}&\longrightarrow & J^s \mathcal{E}\end{array}\end{equation} for $s\leq r$. Each $J^r\mathcal{E}$ contains the subsheaf spanned by the equivalence classes of sections of $\mathcal{E}$ with the same $r$-jets, and so there are morphisms of $\underline{\mathbb{C}}_M$-sheaves: \begin{equation}\label{jr} \begin{array}{ccccc} j^r&:& \mathcal{E} & \longrightarrow & J^r\mathcal{E} \end{array}\end{equation}

Then the linearization $\alpha$ induces isomorphisms $j^2(\alpha_\gamma):j^2(\mathcal{O}_\mathbb{H})\longrightarrow \gamma^* j^2(\mathcal{O}_\mathbb{H})$, and the action of $\gamma$ gives a natural linearization by differentials on $\Omega^k_\mathbb{H}$ so there is a natural linearization $J^2(\alpha)$ on $J^2\mathcal{O}_{\mathbb{H}}$. By construction  $J^2\mathcal{L}$ is the sheaf $(\overline{p}_* J^2 \mathcal{O}_{\mathbb{H}})^{J^2\alpha}$ corresponding to the  linearization  (see \autoref{linearization}) $J^2\alpha$ induced by $\alpha$ on $J^2 \mathcal{O}_{\mathbb{H}}$. Consider the natural trivialization $$\Psi: J^2\mathcal{O}_{\mathbb{H}}\overset{\sim}{\longrightarrow} \mathcal{O}_{\mathbb{H}}^{\oplus 3}$$ given by the global basis $(1\otimes dz_1^{\otimes 2},j^1(1\otimes dz_1), j^2(1))$. Then for any $\gamma\in \Gamma$, $J^2\alpha_\gamma$ is the isomorphism given by the commutative diagram: \begin{equation}\label{jetlinearization} \xymatrix{ J^2\mathcal{O}_{\mathbb{H}} \ar[r]^{\sim}_{\Psi} \ar[d]_{J^2\alpha_\gamma} & \mathcal{O}_\mathbb{H}^{\oplus 3} \ar[d] &\zeta \ar@{|->}[d]  \\ \gamma^* J^2 \mathcal{O}_\mathbb{H} \ar[r]_{\sim}^{\gamma^* \Psi} & \gamma^*\mathcal{O}_{\mathbb{H}}^{\oplus 3} &   (\zeta\circ \gamma^{-1}) \, ^tA_2(z_1) } \end{equation} where: $$ A_2(z_1):=\begin{pmatrix}\frac{1}{(c_\gamma z_1+d_\gamma)^5} & -\frac{3c_\gamma}{(c_\gamma z_1+d_\gamma)^4} & \frac{2c_\gamma^2}{(c_\gamma z_1+d_\gamma)^3} \\ 0& \frac{1}{(c_\gamma z_1+d_\gamma)^3} & -\frac{c_\gamma}{(c_\gamma z_1+d_\gamma)^2} \\ 0&0& \frac{1}{c_\gamma z_1+d_\gamma} \end{pmatrix}$$ Similarly, $J^1\mathcal{L}$ can be described by the linearization corresponding to the lower right minor $A_1(z_1)$ of $A_2(z_1)$ as above. 

Finally, the equivalence between linearized sheaves and sheaves on the base gives a bijection between morphisms of locally free modules $\delta: J^1\mathcal{L}(* \overline{C})\longrightarrow J^2\mathcal{L}(* \overline{C})$ (where $\overline{C}$ is the quotient of some $\Gamma$-invariant divisor $C$ on $\mathbb{H}$) and $\Gamma$-equivariant morphisms of locally free modules $\tilde{\delta}:J^1\mathcal{O}_{\mathbb{H}}(* C) \longrightarrow J^2 \mathcal{O}_{\mathbb{H}}(*C)$. The former morphisms $\delta$ are called \textit{meromorphic differential operators of order two} on $\mathcal{L}$. The set $\mathcal{P}_{\mathcal{L}}$ of such objects is a $\mathcal{O}_{\hat{N}}(*C)$-affine space in the sense that it is the sum of an element and a $\mathcal{O}_{\hat{N}}(*C)$-vector space, where $\mathcal{O}_{\hat{N}}(*C)$ stands for the field of meromorphic functions with poles supported at $C$.

\begin{definition}\label{Op+} Let $\mathcal{L}$ and $\mathcal{P}_{\mathcal{L}}$ as above. \begin{enumerate} \item $\mathcal{P}_{\mathcal{L},+}$ is the subspace consisting of the $\delta \in \mathcal{P}_\mathcal{L}$ with the property: \begin{equation}\label{propertyOp+}   \delta(ker\, \pi_{1}^0)\subset ker\, \pi_{2}^{0}(*C) \end{equation} where $\pi_r^s$ is the truncation map \eqref{truncation}. Explicitely, the subspace $\mathcal{P}_{\mathcal{L},+}$   is the subset of the operators $\delta \in \mathcal{P}_{\mathcal{L}}$ with the property that the matrix of $\tilde{\delta}=\overline{p}^* \delta$ in the canonical basis of $J^1\mathcal{O}_\mathbb{H}$ and $J^2 \mathcal{O}_\mathbb{H}$ is of the form: \begin{equation}\label{matrixop} \Delta(z_1)=\begin{pmatrix} b(z_1) & c(z_1) \\ \nu(z_1) & a(z_1) \\  0 & \mu(z_1) \end{pmatrix} \end{equation}
\item $\mathcal{P}_{\mathcal{L},++}$ is the subspace of the elements $\delta$ of $\mathcal{P}_{\mathcal{L},+}$ such that $\mu=\nu$ in \eqref{matrixop}.
\item  We also define the subspace $\mathcal{P}_{\mathcal{L},0}\subset \mathcal{P}_{\mathcal{L},++}$ of the elements $\delta$ such that the induced morphisms $J^2\mathcal{L}\longrightarrow J^2\mathcal{L}/ker\, \pi_2^0(*C)$ and $ker\, \pi_1^0 \longrightarrow (ker\, \pi_{2}^0/ker\, \pi_2^1)(*C)$ are zero. Equivalently $\mu=\nu=0$ in \eqref{matrixop}.  \end{enumerate}
\end{definition}

\begin{proposition}\label{isoAop} Consider the subsapce $\mathcal{A}^+_{\hat{N},C}$ of meromorphic affine connections on $\hat{N}$, with poles at $C$, satisfying \eqref{codim2sol} and $\mathcal{P}_{\mathcal{L},++}$ as in \eqref{Op+}. Identify elements of $\mathcal{A}^+_{\hat{N},C}$ with the matrices $(f_{ij},g_{ij})_{i,j=1,2}$ of their pullbacks to $\tilde{M}$ and elements of $\mathcal{P}_{\mathcal{L},++}$ with the matrices of their pullbacks to $\tilde{M}$ as in \eqref{matrixop}.  Then the map: $$\begin{array}{ccccccc}\Psi&:&\mathcal{A}^+_{\hat{N},C}  
 & \longrightarrow & \mathcal{O}_{\hat{N}}(*C) &\times &\mathcal{P}_{\mathcal{L},++} \\ &&\hat{\nabla} & \mapsto & ( g_{11} &,& \begin{pmatrix} f_{22} & f_{21} \\ g_{22}-g_{11} & -\frac{1}{3}f_{11} \\ 0 &  g_{22}-g_{11} \end{pmatrix} ) \end{array}$$ is an isomorphism of $\mathcal{O}_{\hat{N}}(*C)$-affine spaces.
\end{proposition}
\begin{proof} By the equivalence of categories between equivariant sheaves and sheaves on the base, there is a bijection between the elements of $\mathcal{P}_{\mathcal{L},++}$ and the matrices $\Delta$ of the form \eqref{matrixop} satisfying $$A_2(z_1)\Delta(\gamma z_1) A_1^{-1}(z_1) = \Delta(z_1)$$ A computation shows that in this case $a,b,c,\nu$ satisfy the same functional equations as the ones satisfied by $-\frac{1}{3}f_{11},f_{22},f_{21}, g_{22}-g_{11}$ where $f_{ij},g_{ij}$ are any solutions of  \eqref{invariancehyperbolic1} and \eqref{invariancehyperbolic2}. Moreover, the subset of solutions $(f_{ij},g_{ij})_{i,j=1,2}$ of the later system satisfying \eqref{codim2sol} is in bijection with pairs consisting of any meromorphic function $g_{11}$, and functions $-\frac{1}{3}f_{11},f_{22},f_{21}, g_{22}-g_{11}$ as before. 
\end{proof}

\begin{lemma}\label{op2} Let $\mathcal{P}_{\mathcal{L},++}$ be the $\mathcal{O}_{\hat{N}}(*C)$-vector space  as in \autoref{Op+}. It contains an element $\delta_1 \in \mathcal{P}_{\mathcal{L},++}\setminus \mathcal{P}_{\mathcal{L},0}$. In particular, it is a direct sum: \begin{equation}\label{directsumop}\Phi : \mathcal{P}_{\mathcal{L},++} \overset{\sim}{\longrightarrow} \mathcal{P}_{\mathcal{L},0} \oplus \mathcal{O}_{\hat{N}}(*C)\delta_1\end{equation}  where the isomorphism $\Phi$ is the projection on $\mathcal{P}_{\mathcal{L},0}=ker(\nu)$ parallel to $\delta_1$.
\end{lemma}
\begin{proof} $\mathcal{P}_{\mathcal{L},++}\setminus \mathcal{P}_{\mathcal{L},0}$ contains the hyperplane $\mathcal{P}_{\mathcal{L},1}=\{\nu=1\}$, which is the subset of elements $\delta$ satisfying: $$\pi_1^0 \circ \delta \circ j^1= Id_{\mathcal{L}}$$ These are exactly the splitting of the meromorphic one jet sequence of $J^1\mathcal{L}$, i.e. meromorphic connections on $J^1\mathcal{L}$ with poles at $C$. This in particular includes the meromorphic $SL_2$-opers on $J^1 \mathcal{L}$, namely meromorphic connections $\overline{\nabla}$ inducing the canonical connection  of $det(J^1\mathcal{L})=\mathcal{O}_{\hat{N}}$, and inducing an isomorphism between $ker\, \pi_1^0(*C)$ and $J^1\mathcal{L}/ ker\, \pi_1^0 (*C)$. This subset is in turn known to be in bijection with the nonempty set of meromorphic projective structures on $\hat{N}$ with poles at $C$ (see \cite{BDG}, Theorem 4.7). We thus define $\delta_1$ as any operator corresponding to such an element.
\end{proof}

We obtain:

\begin{corollary}\label{classificationgenus2} Let $\hat{M}$ be a principal elliptic bundle with odd first Betti number over a complex compact curve $\hat{N}$ with genus $g(\hat{N})\geq 2$. Let $C$ be an effective divisor of $\hat{N}$ such that $\mathcal{O}$ the $\mathcal{O}_{\hat{N}}(*C)$. Then there exists a $\mathcal{O}_{\hat{N}}(*C)$-affine subspace of codimension 3 in the space of meromorphic affine connection on $\hat{M}$ with poles at $C$, which is isomorphic to the $\mathcal{O}_{\hat{N}}(*C)$-affine space: $$\mathcal{O}_{\hat{N}}(*C)^2 \times \mathcal{P}_{\hat{N},C}$$ where $\mathcal{P}_{\hat{N},C}$ is the affine space of meromorphic projective structures on $\hat{N}$ with poles at $C$. 
\end{corollary}
\begin{proof} The assertion follows from the successive application of \autoref{isoAop} and \autoref{op2}, and the fact that $\mathcal{P}_{\mathcal{L},0}$ is isomorphic to the $\mathcal{O}_{\hat{N}}(*C)$-vector space directing the space of meromorphic projective structures on $\hat{N}$ with poles at $C$ as pointed out in the proof of \autoref{op2}.
\end{proof}

\subsection{Quotients  of principal elliptic bundles over higher genus curves}

We now  classify the minimal meromorphical affine surfaces with $a(M)=1$ such that the associated finite cover $\hat{M}$ is a principal elliptic bundle with odd first Betti number over a compact curve with genus $g(\hat{N})\geq 2$.

For, we first recall the geometric description of $\hat{M}$ given in \cite{Klingler}. Let $\mathbb{P}^1$ seen as the homogeneous complex manifold $G/P$, where $G=SL_2(\mathbb{C})$ and $P$ the subgroup stabilizing the line $\mathbb{C}\mathfrak{e}_1\subset \mathbb{C}^2$ through the standard representation ($(\mathfrak{e}_1,\mathfrak{e}_2$ is the canonical basis of $\mathbb{C}^2$). Let $\Gamma'\subset SL_2(\mathbb{R})$ be the image of the holonomy representation of a uniform $(G,G/P)$ structure on $\hat{N}$, that is : $$\overline{p}:\mathbb{H}\longrightarrow\hat{N}=\Gamma'\backslash \mathbb{H}$$ is the universal cover. 

Let us introduce a notation. If $p:E\longrightarrow M$ is a holomorphic $P$-principal bundle and $\rho:P\longrightarrow GL(\mathbb{V})$ a $P$-representation, we let: \begin{equation}\label{repmodule} E(\mathbb{V})=(p_* \mathcal{O}_E\otimes \mathbb{V})^P \end{equation} where the action of $P$ on $p_* \mathcal{O}_E\otimes \mathbb{V}$ is given by $$b\cdot p_* (f\otimes A) = p_* (f\circ b^{-1} \otimes \rho(b)(A))$$  Then we have a natural isomorphism (see for example \cite{Snow}): \begin{equation}\label{OP(1)} \mathcal{O}_{\mathbb{P}^1}(1) \simeq G(\mathbb{C}\mathfrak{e}_1) \end{equation} where $G$ is seen as the total space of the holomorphic $P$-principal bundle $p_{G/P}:G\longrightarrow G/P$. In the rest of the paper we will identify these two modules. In particular there is a natural left $G$-linearization (see \autoref{linearization}) of this module defined for any $g\in G$, by:  \begin{equation}\label{linearizationOP(1)} \begin{array}{ccccc}\phi_g &:& G(\mathbb{C}\mathfrak{e}_1) & \longrightarrow & g^* G(\mathbb{C}\mathfrak{e}_1) \\ && p_{G/P *}\, \tilde{s} & \mapsto& p_{G/P *}\, \tilde{s}\circ g^{-1}\end{array} \end{equation} Now we can restrict this line bundle to $\mathbb{H}\subset G/P$, and we get a $\Gamma'$-linearisation by considering the isomorphisms $(\phi_\gamma)_{\gamma \in \Gamma'}$ as above. Then: \begin{lemma}\label{linebundleproj} The line bundle $\mathcal{L}$ defined as in \autoref{bundleL} is naturally isomorphic to $\mathcal{L}_1=(\overline{p}_* G(\mathbb{C}\mathfrak{e}_1)|_{\mathbb{H}})^{(\phi_\gamma)_{\gamma\in \Gamma'}}$. \end{lemma}
\begin{proof}  It is sufficient to find a trivialization of $\mathcal{L}_1$ such that the isomorphisms $\phi_\gamma$ identifies with the isomorphisms $\alpha_\gamma$ as in \eqref{linearizationL}. For, recall that there exists a global holomorphic section $\sigma_0: \mathbb{H}\longrightarrow G$ given by: $$\sigma_0(z)=\begin{pmatrix} 1 & 0 \\ z &0 \end{pmatrix}$$ Such a section defines a trivialization of any module obtained as a representation of $G$, in particular: \begin{equation}\label{trivOP1} \begin{array}{ccccc}\psi_{\sigma_0}&:& G(\mathbb{C}\mathfrak{e}_1)|_{\mathbb{H}} & \longrightarrow & \mathcal{O}_{\mathbb{H}} \\ && [(\sigma_0,f\mathfrak{e}_1)] &\mapsto & f \end{array}\end{equation}

Now we have: $$\sigma_0(\gamma z) = \gamma \cdot \sigma_0(z) \cdot \begin{pmatrix} c_\gamma z + d_\gamma & 0 \\ 0 & \frac{1}{c_\gamma z + d_\gamma }\end{pmatrix}$$ This implies the following commutative diagram for any $\gamma\in \Gamma'$: $$\xymatrix{G(\mathbb{C}\mathfrak{e}_1)|_{\mathbb{H}} \ar[d]_{\psi_{\sigma_0}} \ar[r]^{\phi_\gamma} & \gamma^* G(\mathbb{C}\mathfrak{e}_1)|_{\mathbb{H}} \ar[d]^{\gamma^* \psi_{\sigma_0}} \\ \mathcal{O}_\mathbb{H} \ar[r]_{\alpha_\gamma} & \gamma^* \mathcal{O}_\mathbb{H} } $$
\end{proof}
We denote by $R(\mathcal{L})$ the $\mathbb{C}^*$-principal bundle whose fiber over $\hat{y}\in \hat{N}$ is the set of non-zero vectors of the fiber $\mathcal{L}(\hat{y})$. Then it is immediate that $R(\mathcal{L})$ is the quotient of $R(G(\mathbb{C}\mathfrak{e}_1)|_\mathbb{H})$ by the action of $\Gamma'$ corresponding to the isomorphisms $\phi_\gamma$ from \eqref{linearizationOP(1)}. Moreover, we have a natural isomorphism: $$R(G(\mathbb{C}\mathfrak{e}_1)|_\mathbb{H})=G|_{\mathbb{H}}/P^+\subset G/P^+$$ where $P^+$ is the kernel of the representation of $P$ on $\mathbb{C}\mathfrak{e}_1$ (i.e. the unipotent radical of $P$). Through this identification, the action of $\Gamma'$ is the natural left action of $\Gamma'\subset G$ on $G/P^+$ (note that $G/P^+$ is biholomorphic to an open subset of $\mathbb{C}^2\setminus \{0\}$ invariant through $\Gamma'$ for the standard action).

Let $\mathbb{Z}\simeq\Delta\subset \mathbb{C}^*$ be a lattice, identified with a subgroup of the standard torus of $G$ (namely the diagonal elements). Since the right action of $\Delta$ on $G$ and the left action of $\Gamma'$ on $G$ commute, there is an induced  right action of $\Delta$ on $R(\mathcal{L})$ covering the identity on $\hat{N}$, and the quotient map is a unramified cover of the complex manifold $\hat{M}$: \begin{equation}\label{Klinglercover} \xymatrix{ & R(\mathcal{L}) \ar[dl]^{p_\Delta} \ar[dd]^{p_R} \\ \hat{M} \ar[dr]_{\hat{\pi}}  & \\ & \hat{N} } \end{equation} where $p_\Delta$ is the quotient map for the action of $\Delta$.

As a remark, note that this description also gives rise to a geometric description for a holomorphic flat affine connection $\hat{\nabla}_0$ on $\hat{M}$. Indeed, $G/P^+$ identifies equivariantly as an open subset of $\mathbb{C}^2 \setminus \{0\}$, and the action of $G$ preserves the canonical flat affine connection of the affine space $\mathbb{C}^2$. In particular the restriction of this connection to $G/P^+$ is both $\Gamma'$-invariant and $\Delta$-invariant, so applying \autoref{quotientconnection} we get a holomorphic connection $\hat{\nabla}_0$ on $\hat{M}$. 
 In particular, we get the following: \begin{lemma}\label{Gammaprincipalellipticgenus2} Let $\hat{M}\longrightarrow \hat{N}$ be a principal elliptic surface with $g(\hat{N})\geq 2$ and $b_1(\hat{M})$ odd. Let $\overline{q}_R: (G/P^+)|_{\mathbb{H}} \longrightarrow R(\mathcal{L})$ be the map of the quotient by the left action of $\Gamma'\subset G$. Then any automorphism $\varphi$  of $\hat{M}$ lifts through $p_\Delta\circ \overline{q}_R $ as the automorphism of $G/P^+$ corresponding to the left action of an element $A\in SL_2(\mathbb{C})$. 
 
 In particular, the holomorphic affine connection $\hat{\nabla}_0$ constructed above is invariant through any automorphism of $\hat{M}$. \end{lemma}

 \begin{proof} The composition $p_\Delta\circ \overline{q}_R $  is an unramified cover of $\hat{M}$, and any automorphism $\varphi$ of $\hat{\pi}:\hat{M}\longrightarrow \hat{N}$ admits a lift to the total space of this cover $\tilde{\varphi}$. In particular, $\tilde{\varphi}$ normalizes the Galois group of $\overline{q}_R$, that is $\Gamma'$, that is : \begin{equation}\label{normalizeGamma}\forall \gamma \in \Gamma',\exists \gamma'\in \Gamma',\, \tilde{\varphi}\circ \gamma =\gamma'\circ \tilde{\varphi} \end{equation} Such an automorphism covers an automorphism of $\mathbb{H}$, that is the action of some $A_1\in SL_2(\mathbb{R})$. Hence, through the trivialization $G/P^+ \simeq \mathbb{H}\times \mathbb{C}^*$ induced by the section $\sigma_0$ from the proof of \autoref{linebundleproj}, we have: $$\tilde{\varphi}(z,b)=(A_1 \cdot z, \lambda(z)b)$$ for some holomorphic function $\lambda:\mathbb{H}\longrightarrow \mathbb{C}^*$. Then \eqref{normalizeGamma} rewrites as: $$\begin{array}{cccc}&(A_1 \cdot \gamma \cdot z&,&\lambda(\gamma \cdot z) (c_{A_1}\gamma \cdot z + d_{A_1})^{-1}(c_\gamma z+d_\gamma)^{-1} b)\\&&&\\=&(\gamma' \cdot A_1 \cdot z&,& \lambda(z) (c_{\gamma'}A_1\cdot z + d_{\gamma'})^{-1}(c_{A_1} z+d_{A_1})^{-1}b) \end{array}$$ In particular $\gamma'=A_1\gamma A_1^{-1}$ so that, using that $\gamma\mapsto \alpha_\gamma$ is an automorphy factor, $\lambda$ is a $\Gamma'$-invariant holomorphic function, that is a constant. This implies the first assertion. The second one is obtained by applying \autoref{quotientconnection}.
     \end{proof}

The final ingredient is the description of holomorphic projective structures in terms of analytical objects called holomorphic \textit{$SL_2(\mathbb{C})$-opers} (see for example \cite{BDG}). On the model $\mathbb{P}^1=G/P$, the sheaf of one-jets $J^1(\mathcal{L}_{G/P}^*)$, where $\mathcal{L}_{G/P}=G(\mathbb{C}\mathfrak{e}_1)$, is naturally isomorphic to $G(\mathbb{C}^2)$ (it can be seen by considering trivialisations of $G\longrightarrow G/P$ as in the proof of \autoref{linebundleproj}). In particular, it contains $\mathcal{L}_{G/P}$ as a locally free submodule. Since $G(\mathbb{C}^2)$ is by definition the sheaf of sections of a homogeneous bundle on $G/P$,there is a natural linearization for the left action of $G$ (see \autoref{linearization}) on this sheaf of $\mathcal{O}_{G/P}$-modules, denoted by $(\phi^J_g)_{g\in G}$. Moreover, it admits a global trivialization induced by the two $P$-equivariant maps $\tilde{s}_i:G\longrightarrow \mathbb{C}^2$ defined by: $$  \tilde{s}_i(\begin{pmatrix} a & b\\c & d \end{pmatrix})=\begin{pmatrix} a & b \\ c & d\end{pmatrix}^{-1} \mathfrak{e}_i$$ for $i=1,2$. The corresponding flat connection of trivial module $\nabla^{J}_{G/P}$ on $G(\mathbb{C}^2)$ is invariant through the isomorphisms $(\phi^J_g)_{g\in G}$ since the above functions are invariant through these isomorphisms. Moreover, since $G$ is a $SL_2(\mathbb{C})$-reduction of the bundle of basis of $G(\mathbb{C}^2)$, there is a natural isomorphism $\overset{2}{\bigwedge} G(\mathbb{C}^2) \simeq \mathcal{O}_{\mathbb{P}^1}$ and the canonical connection coïncides with the connection induce by $\nabla^{J}_{G/P}$. The key property of $\nabla^J_{G/P}$ is that the induced morphism of line bundles: \begin{equation}\label{inducedmorphismoper} \begin{array}{ccccc}[\nabla^J_{G/P}]&:& \mathcal{L}_{G/P}& \longrightarrow & \mathcal{K}_{\mathbb{P}^1}\otimes G(\mathbb{C}^2)/\mathcal{L}_{G/P}\end{array}  \end{equation} is an isomorphism. This indeed enables to recover that \begin{equation}\label{characterizationL} \mathcal{L}_{G/P}^{\otimes 2}=\mathcal{L}_{G/P}\otimes( G(\mathbb{C}^2)/\mathcal{L}_{G/P})\otimes \mathcal{K}_{\mathbb{P}^1} = \mathcal{K}_{\mathbb{P}^1} \end{equation} that is $\mathcal{L}_{G/P}=\mathcal{O}_{\mathbb{P}^1}(1)$. The restriction of $\nabla^J_{G/P}$ over $\mathbb{H}$ is $\Gamma'$-invariant, so it induces a connection $\hat{\nabla}^J_0$ on: \begin{equation}\label{Eproj} \mathcal{E}:=(\overline{q}_* G(\mathbb{C}^2)|_{\mathbb{H}})^{\Gamma'} \end{equation} Since the $SL_2(\mathbb{C})$-reduction of $G(\mathbb{C}^2)$ is also $\Gamma'$-invariant, we get $\overset{2}{\bigwedge}\mathcal{E}=\mathcal{O}_{\hat{N}}$. By construction, we have the same isomorphism as in \eqref{inducedmorphismoper}, so we get: \begin{equation}\label{Eisjet} \mathcal{E}=J^1(\mathcal{L}^*) \end{equation} and recover that \begin{equation}\label{Lisroot} \mathcal{L}^{\otimes 2}=\mathcal{K}_{\hat{N}}\end{equation} The pair $(\mathcal{E},\hat{\nabla}^J_0)$ is called the \textit{holomorphic $SL_2(\mathbb{C})$-oper corresponding to the projective structure} $\hat{N}=\Gamma'\backslash\mathbb{H}$. 

Using these facts, we can prove:

\begin{proposition}\label{noquotientsprincipalellipticg2} Let $\hat{M}\overset{\hat{\pi}}{\longrightarrow}\hat{N}$ be a principal elliptic bundle over a compact curve of genus $g(\hat{N})\geq 2$. Let $\overline{q}: \hat{N}\longrightarrow N$ be a Galoisian finite cover such that $\hat{M}$ is the pullback of a an elliptic surface $\pi:M\longrightarrow N$. Suppose that the sum of the multiplicities $(m_\alpha)_{\alpha \in I}$ of $\overline{q}$ at the ramification points is a multiple of the degree $k=deg(\overline{q})$, and that the Galois group $\Gamma$ of $\overline{q}$ fixes the ramification points.  Then $\overline{q}$ is an unramified cover. \end{proposition}
\begin{proof} Let $\Gamma$ be the Galois group of $\overline{q}$. We identify $\Gamma$ with a subgroup of $SL_2(\mathbb{R})$ normalizing the holonomy $\Gamma'$ of the uniform projective structure on $\hat{N}$ as described before. By \autoref{Gammaprincipalellipticgenus2}, the action of the Galois group $\Gamma$ of $\overline{q}$ on $\hat{N}$ lifts to a left action of $\Gamma$ on the cover $R(\mathcal{L})\overset{q_\Delta}{\longrightarrow} \hat{M}$ (see \eqref{Klinglercover}), obtained from the natural left action of $\Gamma$ on $G$. In particular, there is an induced $\Gamma$-linearization (see \autoref{linearization}) $(\phi^J_{\epsilon})_{\epsilon \in \Gamma}$ on $J^1(\mathcal{L}^*)$. By construction, the holomorphic $SL_2(\mathbb{C})$-oper $(\mathcal{E},\hat{\nabla}^J_0)=(J^1(\mathcal{L}^*),\hat{\nabla}^J_0)$ (see above) is invariant by this $\Gamma$-linearization.
 
Consider the line bundle \begin{equation}\mathcal{L}'=(\overline{q}_* \mathcal{L})^\Gamma\end{equation} on $N$. It is a submodule of the locally free $\mathcal{O}_N$-module \begin{equation}\label{E'oper}\mathcal{E}'=(\overline{q}_* \mathcal{E})^\Gamma\end{equation}

We  prove the existence of a holomorphic connection $\nabla^J_0$ on $\mathcal{E}'$ with the same properties as in \eqref{inducedmorphismoper}. For, we first prove that the universal cover of $N$ factorizes by $\overline{q}:\hat{N}\longrightarrow N$. Pick a system of generators $\gamma_1,\ldots,\gamma_r$ for $\Gamma$. By \autoref{Gammaprincipalellipticgenus2}, these generators lifts through $p_\Delta\circ \overline{q}_R$ to the left actions of elements $A_1,\ldots,A_r\in SL_2(\mathbb{R})$ on $G/P^+|_\mathbb{H}$. By definition of $\Gamma$, $N$ is the quotient of $\mathbb{H}$ by the subgroup $\Gamma''\subset SL_2(\mathbb{R})$ spanned by the $A_j$ and $\Gamma'$. Suppose that there exists $j\in \{1,\ldots,r\}$ such that the Möbius transformation corresponding to $A_j$ fixes some $z\in \mathbb{H}$. Recall that $\gamma_j$ fixes the fibers of $\overline{p}(z)\in \hat{N}$ in $\hat{M}=R(\mathcal{L})/\Delta$. Since $\gamma_j$ have finite order and $\Delta\simeq \mathbb{Z}$, the induced action on $R(\mathcal{L})$ fixes the fiber of $\overline{p}(z)$. Hence $A_j$ fixes the fiber of $z$ in $G/P^+|_\mathbb{H}$. Thus $A_j=Id$, that is $\Gamma''$ acts freely and properly discontinusly, in other words $\mathbb{H}\longrightarrow \Gamma'' \backslash \mathbb{H}=N$ is the universal cover. Since the $A_j$ normalizes $\Gamma'$, we have a factorization: \begin{equation}\label{factorizationuniversal}\xymatrix{ \mathbb{H} \ar[rr]^{\overline{q}''} \ar[rd]^{\overline{p}} & & N \\ & \hat{N} \ar[ru]^{\overline{q}} } \end{equation} where $\overline{q}''$ is the quotient by $\Gamma''$. In particular, $\mathcal{E}'= (\overline{q}''_* G(\mathbb{C}^2)|_\mathbb{H})^{\Gamma''}$ and the quotient of $G|_\mathbb{H}$ by $\Gamma''$ is a holomorphic $P$-reduction of its bundle of basis.

Moreover, the holomorphic connection $\nabla^J_{G/P}$ is invariant by the action of $\Gamma''$, and since $\overline{q}''$ is an unramified cover, \autoref{quotientconnection} implies the existence of an induced holomorphic connection ${\nabla'}^J$ on $\mathcal{E}'$. By construction, the induced morphism: $$[{\nabla'}^J]: \mathcal{L}'\longrightarrow \Omega^1_N \otimes \mathcal{E}'/\mathcal{L}'$$ is a non-trivial morphism.  We denote by $S$ the effective divisor corresponding to the line bundle $End(\mathcal{L}',\mathcal{K}_N \otimes \mathcal{E}/\mathcal{L}')$. 

In the sequel, we employ the notation $deg(\mathcal{L})=\int_{N} c_1(\mathcal{L})$ for any line bundle $\mathcal{L}$ on a compact complex curve $N$.  Using  $deg(\mathcal{E}')=0$ (by the existence of a $SL_2(\mathbb{C})$-holomorphic reduction of its bundle of basis), we get: \begin{equation}\label{degS}deg(S)=deg(\mathcal{K}_{N})-2deg(\mathcal{L}')\end{equation}  
In one other hand, consider the sheaf of sections of the pullback line bundle, that is $\mathcal{O}_{\hat{N}}\otimes \overline{q}^* \mathcal{L}'$. Since $\overline{q}^* \mathcal{L}'$ is by definition the subsheaf of $\mathcal{L}$ spanned by the sections invariant by the action of $\Gamma$, we have a well-defined non-trivial morphism of modules: \begin{equation}\label{iotapullback}\begin{array}{ccc}\iota :\mathcal{O}_{\hat{N}}\otimes \overline{q}^* \mathcal{L}' & \longrightarrow & \mathcal{L}\\ f\otimes s & \mapsto & fs \end{array}\end{equation} Recall that the action of $\Gamma$ on $\hat{M}=\Delta\backslash R(\mathcal{L})$ fixes the fibers of the ramification locus of $\overline{q}$. Hence, in a neighborhood $U_\alpha$ of any component $D_\alpha$ of the ramification locus of $\overline{q}$ in $\hat{N}$, we can find a coordinate $z$ and a trivialization of $\mathcal{L}$, such the action of the automorphism $\epsilon \in \Gamma$ corresponding to a generator of $\pi_1(U_\alpha\setminus D_\alpha,\hat{y})$,  on a section $s\in \mathcal{L}(U_\alpha)$, is given by: $$\epsilon \cdot s(z) = \nu_\epsilon s\circ \epsilon^{-1}(z)$$ for some $\nu_\epsilon \in \Delta$. But $\epsilon$ has finite order $m_\alpha$ and $\Delta$ contains no non-trivial cyclic element, so $\nu=1$. As a consequence, the section corresponding to $1$ in the choosen trivialisation of $\mathcal{L}$ is a local invariant section of $\mathcal{L}$ on $U_\alpha$. This in turn implies that \autoref{iotapullback} is an isomorphism.  Hence, since $\overline{q}$ is a finite cover of degree $k$, we get: \begin{equation}\label{degreeLL'}   deg(\mathcal{L})= deg(\mathcal{O}_{\hat{N}}\otimes q^* \mathcal{L}')=k deg(\mathcal{L}')  \end{equation}
Now, by the Riemann-Hurwitz formula, we also have:  \begin{equation}\label{RiemannHurwitz} -deg(\mathcal{K}_{\hat{N}})=-kdeg(\mathcal{K}_N)+|I|-\underset{\alpha \in I}{\sum}m_\alpha\end{equation} By \eqref{degreeLL'}, and \eqref{degS}, we also have: $$ \begin{array}{ccc}-deg(\mathcal{K}_{\hat{N}})&=&-2deg(\mathcal{L})\\&&\\&=& -2k \, deg(\mathcal{L}') \\&&\\&=& -k\, deg(\mathcal{K}_{N}) + k\, deg(S) \end{array}$$  Combining the above equality with \eqref{RiemannHurwitz} we get: $$ k\, deg(S)+\underset{\alpha\in I}{\sum}m_\alpha = |I|$$ By the assumption on the degree $k$ of $\overline{q}$ the above equality rewrites as: $$|I|=k' \underset{\alpha\in I}{\sum}m_\alpha$$ with $k'\geq 0$ and $m_\alpha\geq 2$ for any $\alpha \in I$. This is possible only if $k'=0$, i.e $I=\emptyset$ and $\overline{q}$ is unramified.

%Alternative: prove deg(S)=0
\end{proof}

\begin{theorem}\label{quotientsprincipalellipticgenus2} Let $(M,\nabla)$ be a minimal meromorphic affine surface with $a(M)=1$ and $(\hat{M}\overset{\hat{\pi}}{\longrightarrow}\hat{N},\hat{\nabla})$ the meromorphic affine principal elliptic bundle obtained as in \autoref{finitecover}. If the genus $g(\hat{N})\geq 2$, then $\hat{M}=M$. 

\end{theorem}
\begin{proof} By \autoref{noquotientsprincipalellipticg2}, $\hat{M}$ has an odd first Betti number, so it corresponds to a quotient of $R(\mathcal{L})$ by a lattice $\mathbb{Z}\simeq \Delta\subset \mathbb{C}^*$ as in \eqref{Klinglercover}.   By construction, the Galois cover $\overline{q}:\hat{N}\longrightarrow N$ (resp. $q:\hat{M}\longrightarrow M$) in \eqref{finitecover} is a composition: $$\overline{q}=\overline{q}'\circ \overline{q}_1\, (\text{resp.} q=q'\circ \overline{q}_1)$$ where $\overline{q}':\hat{N}_1\longrightarrow N$ is a composition of cyclic covers and $\overline{q}_1$ is an unramified finite cover (the maps $q'$ and $q_1$ are the corresponding pullbacks of elliptic bundles).

Denote by $\Gamma_1\subset \Gamma$ the Galois group of $\overline{q}_1$, identified with the Galois group of $q_1$. Then  by construction $\hat{M}_1=\Gamma_1\backslash \hat{M}$ is the quotient of $R(\mathcal{L}_1)$ by $\Delta$ where $$\mathcal{L}_1=(\overline{q}_{1*}\mathcal{L})^{\Gamma_1}$$ and by \autoref{Gammaprincipalellipticgenus2}, the action of $\Gamma_1$ lifts to the natural action of a subgroup  $\tilde{\Gamma}_1\subset SL_2(\mathbb{C})$ on $G(\mathbb{C}\mathfrak{e}_1)|_{\mathbb{H}}$. Hence $\Gamma''=\langle \Gamma',\Gamma_1\rangle$ is the holonomy group of a uniform $(G,G/P)$-structure $$\overline{p}_1:\mathbb{H} \longrightarrow \hat{N}_1=\Gamma''\backslash \mathbb{H}$$ and $$\mathcal{L}_1=(\overline{p}_{1*}G(\mathbb{C}\mathfrak{e}_1)|_\mathbb{H})^{\Gamma''}$$ so $\hat{M}_1$ is obtained as in \eqref{Klinglercover}. In particular, it is a principal elliptic bundle over $g(\hat{N}_1)\geq 2$ with odd first Betti number. Without loss of generality we can and will further assume that $\hat{N}_1=\hat{N}$ and $\hat{M}_1=\hat{M}$. 

In this situation, by construction, the Galois group of  $\overline{q}$ fixes its ramification locus, and $k=deg(\overline{q})$ is a multiple of $\underset{\alpha\in I}{\sum}m_\alpha$. Then \autoref{quotientsprincipalellipticgenus2} proves that $\overline{q}:\hat{N}\longrightarrow N$ is unramified.  Since moreover $\overline{q}=\overline{q}'$, we get $\hat{M}=M$ by definition of $\overline{q}'$. \end{proof}

 \section{Conclusion}

 Comparing with \cite{InoueKobayashi}, the \autoref{reductionprincipalelliptic} together with \autoref{quotientsHopf}, \autoref{classificationsecondary}, \autoref{quotientstori}, and \autoref{noquotientsprincipalellipticg2} prove \autoref{nonew}. From the point of view of the uniformization, this suggests that we have to consider other meromorphic geometric structures to encompass more complex compact surfaces. It could be interesting to extend the technics of this paper to meromorphic projective connections (see \cite{KobayashiOchiai}) by adopting the Cartan point of view.